\documentclass[13pt, reqno]{amsart}
\usepackage{dsfont}
\usepackage{amsmath}
\usepackage{amscd}
\usepackage{amsthm}
\usepackage{amssymb} \usepackage{latexsym}
\usepackage{eufrak}
\usepackage{euscript}
\usepackage{epsfig}
\usepackage{graphics}
\usepackage{array}
\usepackage{enumerate}
\usepackage{color}
\usepackage{wasysym}
\usepackage{pdfsync}
\usepackage{stmaryrd}
\usepackage{graphicx}
\usepackage{url}
\usepackage{float}
\usepackage[font=small, labelformat=simple]{caption}

\newcommand{\bel}[1]{\begin{equation*}\label{#1}}
	
	\newcommand{\be}{\begin{equation}}

		\newcommand{\ba}{\begin{eqnarray}}
			\newcommand{\ea}{\end{eqnarray}}

		\newcommand{\qe}{\end{equation}}
	\newcommand{\R}{{\mathbb R}}
	
	\newcommand{\Z}{{\mathbb Z}}
	\newcommand{\C}{{\mathbb C}}
	\newcommand{\T}{{\mathbb T}}
	\newcommand{\supp}{{\mathrm{supp}}}

	\newcommand{\eg}{\begin{example}}
		\newcommand{\egd}{\end{example}}
	\newcommand{\tm}{\begin{thm}}
		\newcommand{\tmd}{\end{thm}}
	\newcommand{\co}{\begin{coro}}
		\newcommand{\cod}{\end{coro}}
	\newcommand{\enu}{\begin{enumerate}}
		\newcommand{\enud}{\end{enumerate}}
	\newcommand{\rmk}{\begin{rem}}
		\newcommand{\rmkd}{\end{rem}}
	
	\theoremstyle{theorem}
	\newtheorem{thm}{Theorem}[section]
	\newtheorem{prop}[thm]{Proposition}
	\theoremstyle{example}
	\newtheorem{example}[thm]{Example}
	\newtheorem{coro}[thm]{Corollary}
	\theoremstyle{lemma}
	\newtheorem{lemma}[thm]{Lemma}
	\theoremstyle{definition}
	\newtheorem{defi}[thm]{Definition}
	\theoremstyle{proof}
	
	\theoremstyle{remark}
	\newtheorem{rem}[thm]{Remark}
	\theoremstyle{remark}
	
	\newtheorem{conj}[thm]{Conjecture}
	
	\UseRawInputEncoding

\begin{document}

		\title[Restriction Estimates for 2D Surfaces of Finite Type 3 and Applications to the Dispersive Equations]{Restriction Estimates for 2D Surfaces of Finite Type 3 and Applications to Dispersive Equations}

		\author{Jiajun Wang}
		\address{Jiajun Wang: Courant Institute of Mathematical Sciences,
			New York University, New York, NY}
		\email{jw9409@nyu.edu}
		
		\begin{abstract}
			In this paper, we prove the restriction estimates for 2D surfaces 
			$S:=\lbrace (\xi_{1},\xi_{2},\xi_{1}^{3}\pm\xi_{2}^{3}): (\xi_{1},\xi_{2})\in [0,1]^{2} \rbrace$ 
			by reducing to Wang--Wu's result on the perturbed paraboloid and to the results on the perturbed hyperboloid obtained by Buschenhenke, M\"{u}ller, and Vargas, as well as by Guo and Oh. The method is based on the rescaling technique developed in \cite{15}. Besides, we will use the estimates to give a better analysis for discrete nonlinear Schr\"{o}dinger equations.
		\end{abstract}

		\maketitle
		\numberwithin{equation}{section}
		\section{Introduction}
	    In harmonic analysis, Fourier restriction is a very vibrant field that has deep connections with other areas such as dispersive equations, incidence geometry, and number theory.
	    
		The Fourier restriction problem for a hypersurface S in $\mathbb{R}^{d}$ asks for the range of $(p,q)$, such that
		\begin{equation*}
			\left(\int_{S}|\widehat{f}|^{p}d\sigma\right)^{\frac{1}{p}}\le C\|f\|_{L^{q}(\mathbb{R}^{d})}, \quad \forall f\in \mathcal{S}(\mathbb{R}^{d}),
		\end{equation*}
		where $d\sigma$ is surface measure and constant $C=C(p,q,S)$ is independent of $f$.
		
		This kind of problem is first proposed by Stein in \cite{39}. After that, Fefferman, Stein, and Zygmund thoroughly solved this problem for curves with non--vanishing curvature in 2D (see \cite{40,41}). In higher dimensions, Tomas and Stein first proved the $L^{q}\to L^{2}$ estimate for hypersurfaces with non--vanishing Gaussian curvature in \cite{20,43}.
		
		For convenience, we prefer to use its dual form, in terms of the extension operator.
		\begin{equation}\label{extension}
			\|E_{S}g\|_{L^{p}(\mathbb{R}^{d})}\le C\|g\|_{L^{q}(S,d\sigma)},
		\end{equation}
		where the extension operator is defined as
		\begin{equation*}
			E_{S}g(x):=\int_{S}e^{ix\cdot\omega}g(\omega)d\sigma.
		\end{equation*}
		
		Using the $\varepsilon$--removal lemma (see \cite{48,51}), Stein's conjecture in 3D can be stated in the following local form.
		
		\begin{conj}\label{conj1}
			If $S$ is a smooth surface in $\R^{3}$ with positive definite second fundamental form, then we have for any $\varepsilon>0$,
			\begin{equation*}
				\|E_{S}g\|_{L^{p}(B_{R}^{3})}\lesssim_{\varepsilon}R^{\varepsilon}\|g\|_{L^{\infty}(S,d\sigma)},
			\end{equation*}
			where $p>3$ and $B_{R}^{3}\subseteq\R^{3}$ denotes a ball centered at the origin with radius $R$ .
		\end{conj}
		
		To obtain a general $L^{q}\to L^{p}$ estimate, many important tools have been introduced. Notably, Bourgain introduced wave packet decomposition and induction on scale in \cite{44,45}. Moyua, Vargas, Vega, and Tao established the bilinear approach to study the restriction estimates (see \cite{46,47,48}). Applying the multilinear restriction estimate,  Bourgain and Guth established bounds for extension operators and more general H\"{o}rmander-type oscillatory integral operators in \cite{65}. Guth innovatively introduced the polynomial partitioning method and obtained the 3D estimate for the perturbed paraboloid when $p>3.25$ (see \cite{50}). Recently, Wang and Wu used refined decoupling theorems and two--ends Furstenberg inequalities to study the restriction problem, and improved the 3D estimate to $p>22/7$ (see \cite{49}).
		
		Naturally, we may ask for the range of restriction estimate for certain degenerate surfaces, i.e., those whose Gaussian curvature may vanish somewhere. In this direction, Ikromov, Kempe, and M\"{u}ller derived a sharp Tomas--Stein type estimate for the finite-type surfaces using oscillatory integral theories (see \cite{11,9,10}). Buschenhenke, M\"{u}ller, and Vargas utilized the bilinear method and proved a general restriction estimate for the finite-type surfaces (see \cite{6}). Recently, Li, Miao, and Zheng developed the rescaling technique and gave a better estimate for the finite-type surfaces $F_{4}^{2}:=\lbrace(\xi_{1},\xi_{2},\xi_{1}^{4}+\xi_{2}^{4}):(\xi_{1},\xi_{2})\in [0,1]^{2}\rbrace$ in \cite{15}.
		
		In this paper, we focus on surfaces of finite type $3$, with the prototypical case $S:=\lbrace(\xi_{1},\xi_{2},\xi_{1}^{3}\pm\xi_{2}^{3}):(\xi_{1},\xi_{2})\in [0,1]^{2}\rbrace$ or, equivalently, $S:=\lbrace(\xi_{1},\xi_{2},\xi_{1}^{3}+\xi_{2}^{3}):(\xi_{1},\xi_{2})\in [-1,1]^{2}\rbrace$. Besides, the sharp restriction estimate for curves of finite type $3$, with the prototypical case $\gamma:=\lbrace(t,t^{3}):t\in[-1,1]\rbrace$, has been completely proved (see \cite{61,63})
		
		The importance of such restriction estimates for surfaces/curves of finite type $3$ lies in the deep connections with KdV and 2D Zakharov--Kuznetsov equations. 
		
		Recall the KdV and 2D Zakharov--Kuznetsov equations:
		\begin{equation*}
			\textbf{KdV}: \quad \left\{
			\begin{aligned}
				& \partial_{t}u+\partial_{x}^{3}u=\partial_{x}(u^{2}),  \\
				& u(x,0) = u_{0}(x), \quad (x,t)\in\ \R\times \R,
			\end{aligned}
			\right.
		\end{equation*}
		\begin{equation*}
			\textbf{ 2D  Zakharov--Kuznetsov}:\quad \left\{
			\begin{aligned}
				& \partial_{t}u+\partial_{x_{1}}\Delta u=\partial_{x_{1}}(u^{2}),  \\
				& u(x,0) = u_{0}(x), \quad (x,t)\in\ \R^{2}\times \R.
			\end{aligned}
			\right.
		\end{equation*}
		By appropriately changing variables, restriction estimates for surfaces (resp. curves) of finite type $3$ can directly yield space--time homogeneous Strichartz estimates for 2D Zakharov --Kuznetsov equations (resp. KdV equations). Currently, the derivations of Strichartz estimates for those two equations are more or less based on the oscillatory integral theories and the Kenig, Gustavo Ponce, and Luis Vega's foundational work in \cite{57}. It is possible that deriving Strichartz estimates from restriction theories will produce more useful tools for studying dispersive equations. For more discussions on those two equations, we refer to \cite{30,59,60}.
		
		However, in this paper, we do not apply restriction estimates to the KdV or 2D Zakharov--Kuznetsov equations. In fact, we will apply our results to discrete nonlinear Schr\"{o}dinger equations, which will be introduced in the last section.
		 
		 \vspace{5pt}
		Now we give a modified definition for the functions of finite type (see \cite{8}), which will constitute our surfaces/curves of finite type.
		 \begin{defi}
		 	For $\varphi\in C^{\infty}([-1,1])$, we say that $\varphi$ is of finite type $n$, if 
		 	\begin{equation*}
		 		\varphi(0)=\varphi'(0)=\cdots=\varphi^{(n-1)}(0)=0\ne \varphi^{(n)}(0).
		 	\end{equation*}
		 	For convenience, we additionally assume that 
		 	\begin{equation*}
		 		1/2\le \varphi^{(n)}(0)\le 2, \quad \|\varphi^{(n+k)}\|_{C^{0}([-1,1])}\le 10^{-100}, \quad \forall k\ge 1.
		 	\end{equation*}
		 \end{defi}
		 
		And before we present our restriction estimates, we invoke and slightly restate Wang--Wu's restriction estimate for the perturbed paraboloid \cite{49}.
		\begin{thm}\label{WW}
		    Suppose $S:=\lbrace(\xi_{1},\xi_{2},\phi(\xi_{1},\xi_{2})):(\xi_{1},\xi_{2})\in [0,1]^{2}\rbrace$ is any compact, smooth surface in $\R^{3}$, satisfying 
		    \begin{equation*}
		    	\phi(0)=0, \quad \nabla \phi(0)=0, \quad D^{2}\phi(0)=
		    	\begin{bmatrix}
		    		1 & 0 \\
		    		0 & 1
		    	\end{bmatrix},
		    \end{equation*}
		    then for $p>22/7$,
		    we have a restriction estimate, for $C=C(p)>0$ independent of $\phi$,
		    \begin{equation*}
		    	\|E_{S} g\|_{L^{p}(\mathbb{R}^{3})}\le C\|g\|_{L^{p}({[0,1]^{2}})}.
		    \end{equation*}
		\end{thm}
		\begin{rem}
			In \cite{49}, Wang and Wu actually proved the above $L^{p}\to L^{p}$ estimate for any compact $C^{2}$ surface with a strictly positive second fundamental form.  To capture the uniformity in the estimate, we impose a redundant condition on the Hessian. By scaling and rotation, we can apply this theorem to surfaces with strictly positive second fundamental forms. 
		\end{rem}
	
		Our main theorem is stated as follows.
		\begin{thm}\label{main}
			Let $\varphi_{1}, \varphi_{2}$ be of finite type $3$, then for $S:=\lbrace(\xi_{1},\xi_{2},\varphi_{1}(\xi_{1})+\varphi_{2}(\xi_{2})):(\xi_{1},\xi_{2})\in [0,1]^{2}\rbrace$, we have the local restriction estimate
			\begin{equation*}
				\left\|E_{S} g\right\|_{L^{22/7}(B_{R}^{3})}\lesssim_{\varepsilon} R^{\varepsilon}\|g\|_{L^{\infty}([0,1]^{2})}, \quad \forall \varepsilon>0.
			\end{equation*}
		\end{thm}
		\begin{rem}
			To prove this theorem, we will follow the rescaling technique developed in \cite{15} to reduce our discussions to Wang--Wu's result. For brevity, we sometimes refer to such $S$ as the surface of finite type $3$ or finite-type $3$ surface. 
		\end{rem}
		\begin{rem}
			The prototypical example is $S:=\lbrace(\xi_{1},\xi_{2},\xi_{1}^{3}+\xi_{2}^{3}):(\xi_{1},\xi_{2})\in [0,1]^{2}\rbrace$. To visualize our new range of restriction estimate, we refer to  Figure 1. Applying the $\varepsilon$--removal lemma (see \cite{4,5}), we obtain the restriction estimate (\ref{extension}) for ($\frac{1}{q},\frac{1}{p}$) on the segment $\overline{AC}$ (excluding the point $A$). In \cite{6}, Buschenhenke, M\"{u}ller, and Vargas proved the case when ($\frac{1}{q},\frac{1}{p}$) lies in the trapezoid BCDE (excluding the segments $\overline{BE}$ and $\overline{EI}$) using bilinear estimates. Then, by interpolation, we can establish the case in quadrilateral ACDE (excluding the segments $\overline{AE}$ and $\overline{EI}$). Besides, the conjectured range for the prototypical case is the trapezoid GCDF (excluding the segment $\overline{GF}$).
		\end{rem}
	    In order to apply the restriction estimate to dispersive equations, we actually need the restriction estimate on $[-1,1]^{2}$, not only on $[0,1]^{2}$. By symmetry, we only need to establish the ``hyperbolic" case  $S:=\lbrace(\xi_{1},\xi_{2},\varphi_{1}(\xi_{1})-\varphi_{2}(\xi_{2})):(\xi_{1},\xi_{2})\in [0,1]^{2}\rbrace$. Another important motivation for studying this case is that, in \cite{16}, Schwend and Stovall proved the restriction
	    conjecture for degenerate hypersurfaces $S$ be of the following form:
	    \begin{equation*}
	    	S:=\lbrace (\xi_{1},\xi_{2},\cdots,\xi_{d-1},\sum_{j=1}^{d-1}|\xi_{j}|^{\beta_{j}}):(\xi_{1},\xi_{2},\cdots,\xi_{d-1})\in [0,1]^{d-1}\rbrace, \quad \forall \beta_{j}>1, \;1\le j\le d-1,
	    \end{equation*}
	    by assuming Conjecture \ref{conj1}. Thus, it is of great interest to study the ``hyperbolic" case. However, this case is more difficult than the previous one, as the two principal curvatures have different signs. To overcome this difficulty, we invoke a hyperbolic analogue of Theorem~\ref{WW}, established independently by Buschenhenke, M\"{u}ller, and Vargas, and by Guo and Oh (see \cite{8,66}).
	    \begin{thm}\label{Hyper}
	    	 Suppose $S:=\lbrace(\xi_{1},\xi_{2},\phi(\xi_{1},\xi_{2})):(\xi_{1},\xi_{2})\in [0,1]^{2}\rbrace$ is any compact, smooth surface in $\R^{3}$, satisfying 
	    	\begin{equation*}
	    		\phi(0)=0, \quad \nabla \phi(0)=0, \quad D^{2}\phi(0)=
	    		\begin{bmatrix}
	    			0 & 1 \\
	    			1 & 0
	    		\end{bmatrix},
	    	\end{equation*}
	    	then for $p>13/4$, $p>2q'$,
	    	we have a restriction estimate, for $C=C(p)>0$ independent of $\phi$,
	    	\begin{equation*}
	    		\|E_{S} g\|_{L^{p}(\mathbb{R}^{3})}\le C\|g\|_{L^{q}({[0,1]^{2}})}.
	    	\end{equation*}
	    \end{thm}
	    \begin{rem}
	    	Still, by scaling and rotation, we can apply this theorem to surfaces with a strictly negative curvature. 
	    \end{rem}
	    Using this result and the rescaling techniques, we can recover Theorem \ref{main}.
	    \begin{thm}\label{main2}
	    Let $\varphi_{1}, \varphi_{2}$ be of finite type $3$, then for $S:=\lbrace(\xi_{1},\xi_{2},\varphi_{1}(\xi_{1})-\varphi_{2}(\xi_{2})):(\xi_{1},\xi_{2})\in [0,1]^{2}\rbrace$, $\frac{3}{10}\le \frac{1}{p}<\frac{4}{13}$, $\frac{11}{2p}+\frac{3}{4}\le\frac{5}{2q'}$, we have the local restriction estimate
	    \begin{equation*}
	    	\left\|E_{S} g\right\|_{L^{p}(B_{R}^{3})}\lesssim_{\varepsilon} R^{\varepsilon}\|g\|_{L^{q}([0,1]^{2})}, \quad \forall \varepsilon>0.
	    \end{equation*}
	    \end{thm}
	    \begin{rem}
	    Naturally, the prototypical example this time is $S:=\lbrace(\xi_{1},\xi_{2},\xi_{1}^{3}-\xi_{2}^{3}):(\xi_{1},\xi_{2})\in [0,1]^{2}\rbrace$.
	    It is also worth mentioning that we can only obtain the estimates for surfaces such as $\lbrace(\xi_{1},\xi_{2},\xi_{1}^{m_{1}}+\xi_{2}^{m_{2}}):(\xi_{1},\xi_{2})\in[0,1]^{2}\rbrace$ from \cite{6}. Therefore, at present, we have the Tomas--Stein type estimate (the point H in Figure 2) and the trivial $L^{1}\to L^{\infty}$ bound (the point D). Now, by interpolation, the range of $(\frac{1}{q},\frac{1}{p})$ shrinks to the pentagon ACDHK (excluding the segments $\overline{AK}$, $\overline{KH}$). The Tomas--Stein type estimate can be derived from oscillatory integrals, which were extensively studied in \cite{9,10,11}.
	    \end{rem}
		\begin{figure}
			\centering
			\includegraphics[width=0.8\linewidth]{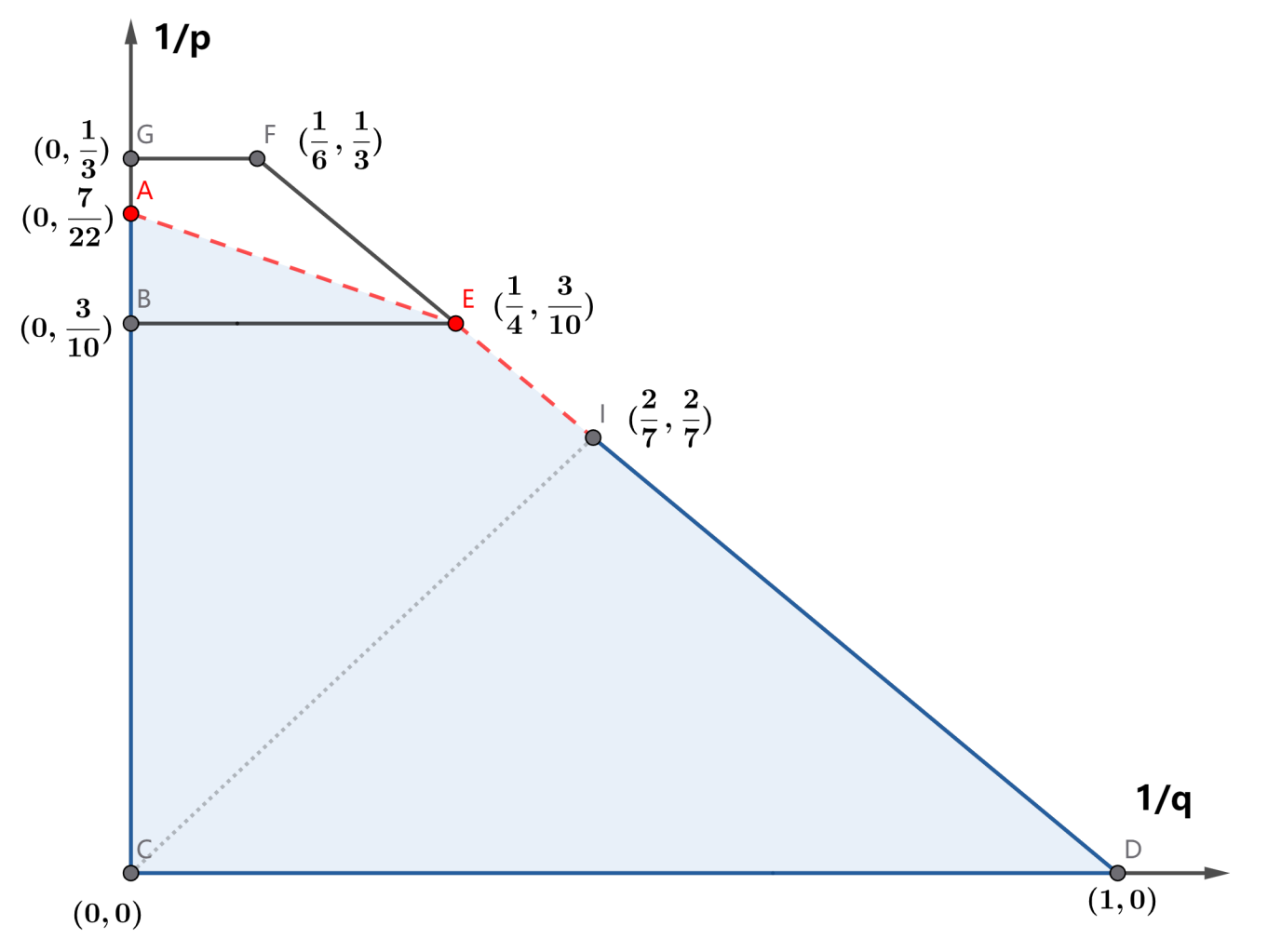}
			\caption{Range for $\varphi_{1}(\xi_{1})+\varphi_{2}(\xi_{2})$}
			\label{fig:restrictionrange1}
		\end{figure}

		\begin{figure}
			\centering
			\includegraphics[width=0.8\linewidth]{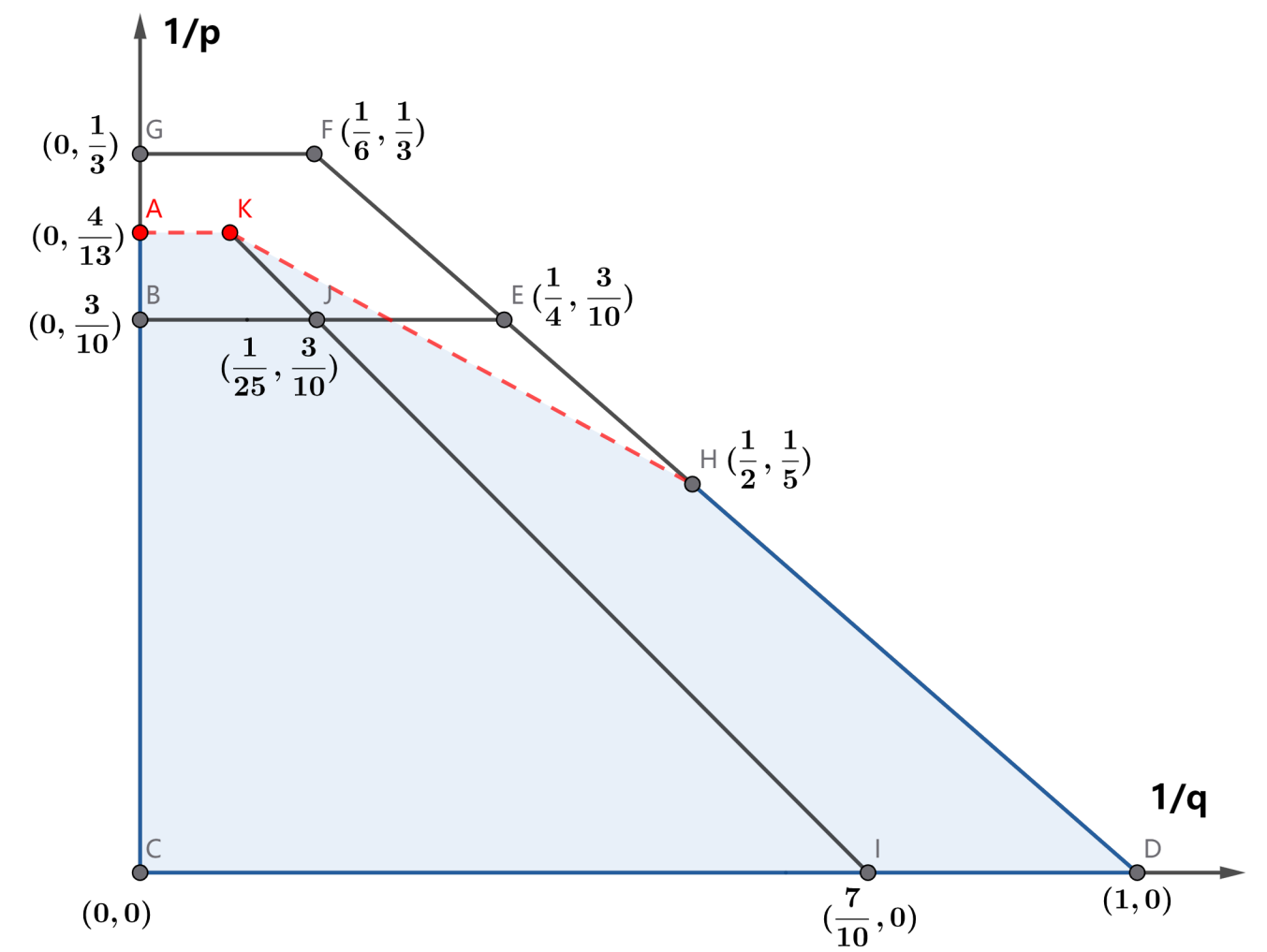}
			\caption{Range for $\varphi_{1}(\xi_{1})-\varphi_{2}(\xi_{2})$}
			\label{fig:restrictionrange2}
		\end{figure}

		For later application in dispersive equations, we recall the following sharp restriction estimate for curves of finite type $3$ (see \cite{61,63}).
		\begin{thm}\label{2D}
			Let $\varphi$ be of finite type $3$, then for curve $\gamma:=\lbrace(t,\varphi(t)):t\in[-1,1]\rbrace$, $0\le \frac{1}{p}<\frac{1}{4}$, $\frac{1}{q}+\frac{4}{p}\le1$, we have the sharp restriction estimate
			\begin{equation*}
				\|E_{\gamma} g\|_{L^{p}(\R^{2})}\lesssim\|g\|_{L^{q}([-1,1])}, \quad \forall\varepsilon>0.
			\end{equation*}
		\end{thm}
		\begin{rem}
			The prototypical example is $\gamma:=\lbrace(t,t^{3}):t\in[-1,1]\rbrace$. The range is the triangle BDC in Figure 3 (excluding the point B). As a simple application, we can use the rescaling technique to give a short proof of the non--endpoint case, i.e., for ($\frac{1}{q},\frac{1}{p}$) not on the segment $\overline{BC}$.
			
		\end{rem}
		\begin{figure}
			\centering
			\includegraphics[width=0.7\linewidth]{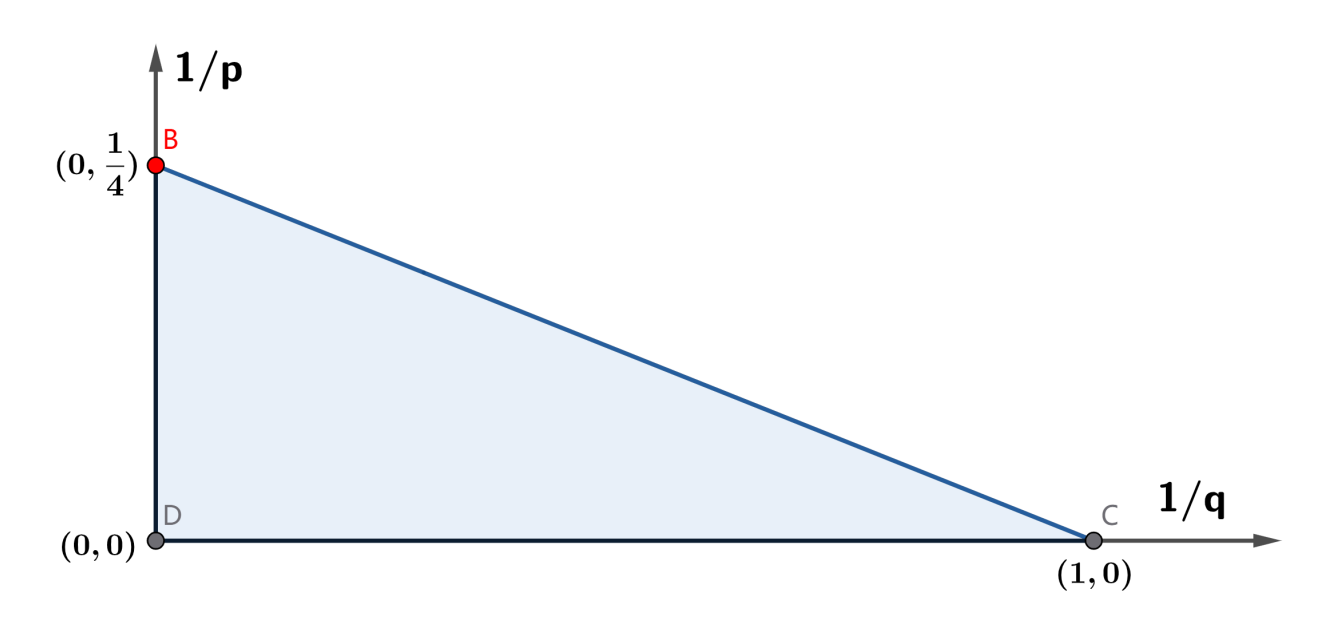}
			\caption{Range for $\varphi(t)$}
			\label{fig:2drange}
		\end{figure}
		
		From now on, we assume that $S:=\lbrace(\xi_{1},\xi_{2},\xi_{1}^{3}+\xi_{2}^{3}):(\xi_{1},\xi_{2})\in [0,1]^{2}\rbrace$, and write $E_{S}$ as $E$ for brevity.  We will see that the proof for the general case or the ``hyperbolic" case is parallel.
		
		Sometimes, we also write $E_{\Theta}$ for some subset $\Theta\subseteq [0,1]^{2}$ to emphasize that the restriction is now on $\lbrace(\xi_{1},\xi_{2},\xi_{1}^{3}+\xi_{2}^{3}):(\xi_{1},\xi_{2})\in \Theta\rbrace$.
        \vspace{7pt}
        
        Now we present the outline of the proof. Note that the finite-type 3 surfaces degenerate only at $\xi_{1}=0$ and $\xi_{2}=0$, we decompose the whole region $[0,1]^{2}$ into four pieces (see Figure 4), where $K$ is a fixed large constant, i.e., $1\ll K\ll R^{\varepsilon}$.

        Next, we will perform the induction on scale. Let $Q(R)$ be the optimal constant such that 
        	\begin{equation}\label{n}
        		\|Eg\|_{L^{22/7}(B_{R}^{3})}
        		\le Q(R)\|g\|_{L^{\infty}([0,1]^{2})}, \quad \forall g\in L^{\infty}([0,1]^{2}).
        	\end{equation}
         We will show, for some constant $C(\varepsilon)>0$, $Q(R)\le C(\varepsilon)R^{\varepsilon}, \forall \varepsilon>0, R>1$. From H\"{o}lder's inequality, the statement is obviously true for $R\lesssim 1$, then we can assume that $Q(R')\le C(\varepsilon)(R')^{\varepsilon}$, $\forall R'<R/2$.

        \begin{figure}
        	\centering
        	\includegraphics[width=0.7\linewidth]{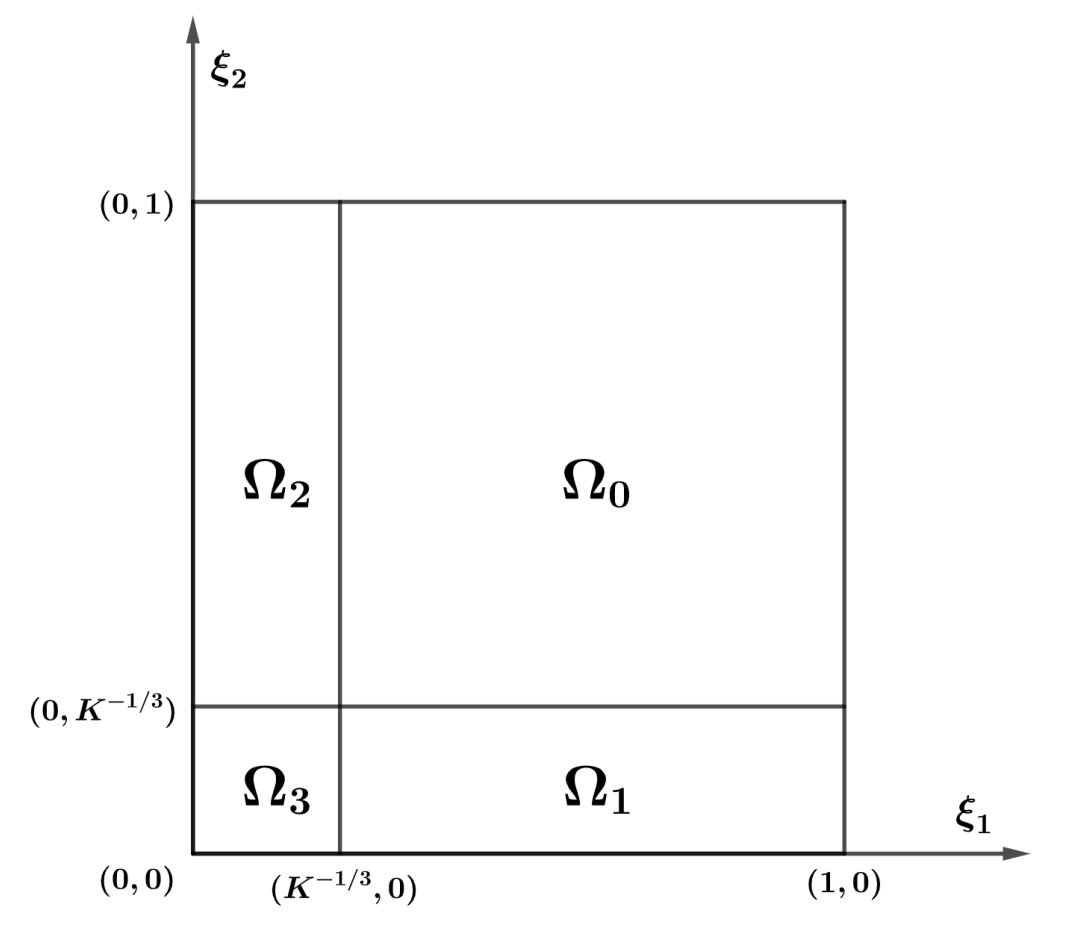}
        	\caption{}
        	\label{fig:omega}
        \end{figure}

        For $\Omega_{0}$, we can directly use Theorem \ref{WW} or Theorem \ref{Hyper} to derive 
        \begin{equation*}
        	\|E g_{\Omega_{0}}\|_{L^{22/7}(B_{R}^{3})}\le C(K)A(\varepsilon)R^{\varepsilon}\|g\|_{L^{\infty}([0,1]^{2})}, \quad g_{\Omega_{0}}:=g\chi_{\Omega_{0}},
        \end{equation*}
        where $C(K)>0$ (resp. $A(\varepsilon)>0$) depends only on $K$ (resp. $\varepsilon$).
        
        For $\Omega_{3}$, we can use the induction on scale to obtain
        \begin{equation*}
        	\|Eg_{\Omega_{3}}\|_{L^{22/7}(B_{R}^{3})}\le C K^{-3/22}Q\left(\frac{R}{K^{1/3}}\right)\|g\|_{L^{\infty}([0,1]^{2})}, \quad g_{\Omega_{3}}:=g\chi_{\Omega_{3}}.
        \end{equation*}
        For the remaining cases, we only need to estimate $\Omega_{1}$ by symmetry. 
        
        We dyadically decompose $\Omega_{1}$ into finer pieces as follows.
        
        	\begin{equation*}
        		\Omega_{1}:=\bigcup_{\lambda}\Omega_{\lambda}, \quad \Omega_{\lambda}:=[\lambda,2\lambda]\times [0,K^{-1/3}],
        	\end{equation*}
        	where $\lambda=2^{j-1}K^{-1/3}$, $1\le j\le \left[\frac{1}{3}\log_{2}(K)\right]$.
        
        In the next section, we will obtain the following estimate:
        \begin{equation*}
        	\|E g_{\Omega_{\lambda}}\|_{L^{22/7}(B_{R}^{3})}\le C(K)A(\varepsilon)R^{\varepsilon}\|g\|_{L^{\infty}([0,1]^{2})}, \quad g_{\Omega_{\lambda}}:=g\chi_{\Omega_{\lambda}}.
        \end{equation*}
        
        Combining the above estimates, we derive that
        \begin{equation*}
        	Q(R)\le C(K)A(\varepsilon)R^{\varepsilon}+CK^{-3/22}Q\left(\frac{R}{K^{1/3}}\right).
        \end{equation*}
        Then, from the induction hypothesis, we have $Q\left(\frac{R}{K^{1/3}}\right)\le C(\varepsilon)\left(\frac{R}{K^{1/3}}\right)^{\varepsilon}$. Then, by taking $K$ sufficiently large, one can easily obtain the desired bound $Q(R)\le C(\varepsilon)R^{\varepsilon}$, for all $\varepsilon>0$ from the above recursive relation.
        
        \vspace{7pt}	
		We organize the paper as follows. In Section 2, we establish a restriction estimate on $\Omega_{\lambda}$ using $\ell^{2}$--decoupling. In addition, we apply the rescaling technique to prove the non--endpoint case of Theorem \ref{2D}. In Section 3, we find interesting connections between our newly established restriction estimate for the finite-type $3$ surfaces and discrete Schr\"{o}dinger equations, and derive better analysis for those equations.

		\noindent
		
		\vspace{5pt}
		\textbf{Notation.}
		\begin{itemize}
			\item By $u\in C^{k}(I; B)$ $($resp.$\; L^{p}(I;B))$ for a Banach space $B$ and time interval $I$, we mean $u$ is a $C^{k}$ (resp.$ \; L^{p})$ map from $I$ to $B$ (see Chapter 5 in \cite{3}).
			\item By $A\lesssim B$ (resp. $A\sim B$), we mean there is a positive constant $C$, such that $A\le CB$ (resp. $C^{-1}B\le A \le C B$). If the constant $C$ depends on $p,$ then we write $A\lesssim_{p}B$ (resp. $A\sim_{p} B$). 
			\item By $\widehat{F}$ (resp. $\widetilde{G}$), we mean the Fourier transform in 2D (resp. 3D).
			\item By $B_{R}^{d}(x)$, we mean the open ball in $\mathbb{R}^{d}$ with center $x$ and radius $R$. 
		\end{itemize}
		\vspace{10pt}
	
		\section{Restriction estimate on $\Omega_{\lambda}$}
		We first present a parallel decoupling lemma (see \cite{2}).
		\begin{lemma}\label{parallel}
			For $p\ge 2$, let $g=\sum_{j}g_{j}$ and $\mu_{i},\omega_{i}$ be measures, with $\mu=\sum_{i}\mu_{i}$, $\omega=\sum_{i}\omega_{i}$, if we have
			\begin{equation*}
				\|g\|_{L^{p}(\mu_{i})}\le D\left(\sum_{j}\|g_{j}\|_{L^{p}(\omega_{i})}^{2}\right)^{1/2}, \quad \forall i, 
			\end{equation*}
			then we have 
			\begin{equation*}
				\|g\|_{L^{p}(\mu)}\le D\left(\sum_{j}\|g_{j}\|_{L^{p}(\omega)}^{2}\right)^{1/2},
			\end{equation*}
			with the same decoupling constant $D$.
		\end{lemma}
		\begin{proof}
			It is a direct application of Minkowski's inequality.
			\begin{equation*}
				\int |g|^{p} d\mu=\sum_{i}\int|g|^{p}d\mu_{i}\le D^{p}\sum_{i}\left(\sum_{j}\|g_{j}\|_{L^{p}(\omega_{i})}^{2}\right)^{p/2}
			\end{equation*}
			\begin{equation*}
				=D^{p}\left\|\left\lbrace\sum_{j}\|g_{j}\|_{L^{p}(\omega_{i})}^{2}\right\rbrace_{i}\right\|_{\ell^{p/2}}^{p/2}\le D^{p}\left[\sum_{j}\left\|\left\lbrace\|g_{j}\|_{L^{p}(\omega_{i})}^{2}\right\rbrace_{i} \right\|_{\ell^{p/2}}\right]^{p/2}
			\end{equation*}
			\begin{equation*}
				=D^{p}\left[\sum_{j}\left(\sum_{i}\|g_{j}\|_{L^{p}(\omega_{i})}^{p}\right)^{2/p}
				\right]^{p/2}=D^{p}\left(\sum_{j}\|g_{j}\|_{L^{p}(\omega)}^{2}\right)^{p/2}.
			\end{equation*}
		\end{proof}
		\begin{lemma}\label{3.3}
			For $2\le p\le 6$, $\gamma_{\lambda}:=\lbrace(t,t^{3}):t\in[\lambda,2\lambda]\rbrace$, $\supp\;   \widehat{F}\subseteq \mathcal{N}_{K^{-1}}(\gamma_{\lambda})$, we have 
			\begin{equation*}
				\|F\|_{L^{p}(B_{K}^{2})}\lesssim_{\varepsilon} K^{\varepsilon} \left(\sum_{\theta}\|F_{\theta}\|_{L^{p}(B_{K}^{2})}^{2}\right)^{1/2}, \quad \widehat{F_{\theta}}=\chi_{\theta}\widehat{F},
			\end{equation*}
			where the sum is over all $\theta: \lambda^{-1/2}K^{-1/2}\times K^{-1}$--slabs contained in $\mathcal{N}_{K^{-1}}(\gamma_{\lambda})$.
		\end{lemma}
		\begin{proof}
			By translation, we can reduce this to the following estimate.
			\begin{equation*}
				\|E_{[\lambda,2\lambda]}f\|_{L^{p}(B_{K}^{2})}\lesssim_{\varepsilon} K^{\varepsilon} \left(\sum_{I:\lambda^{-1/2}K^{-1/2}-\text{interval} }\|E_{I}f\|_{L^{p}(B_{K}^{2})}\right)^{1/2}.
			\end{equation*} 
			Changing variables $t:=\lambda u+\lambda$, we have 
			\begin{equation*}
				\Big|E_{[\lambda,2\lambda]}f(y)\Big|=\Big|\int_{\lambda}^{2\lambda}f(t)e^{i(y_{1}t+y_{2}t^{3})}dt\Big|
			\end{equation*}
			\begin{equation*}
				=\Big|\int_{0}^{1}\widetilde{f}(u)e^{i(\widetilde{y_{1}}u+\widetilde{y}_{2}\gamma(u))}du\Big|\triangleq\Big|E_{[0,1]}^{\gamma}\widetilde{f}(\widetilde{y})\Big|, 
			\end{equation*}
			where we denote 
			\begin{equation*}
				\widetilde{f}(u)=\lambda f(\lambda u+\lambda),\; \gamma(u)=3u^{2}+u^{3},
			\end{equation*}
			with the linear mapping $\mathcal{L}$ in the physical space
			\begin{equation*}
				\mathcal{L}(y)=\mathcal{L}((y_{1},y_{2})):=(\widetilde{y_{1}},\widetilde{y_{2}})=(\lambda y_{1}+3\lambda^{3}y_{2}, \lambda^{3}y_{2}).
			\end{equation*}
			Since $\gamma(u)$ now has positive curvature, we can apply the  $\ell^{2}$--decoupling inequality (see \cite{1}) and the (inverse) Minkowski's inequality to derive
			\begin{equation*}
				\left\|E_{[\lambda,2\lambda]}f\right\|_{L^{p}(B_{K}^{2})}^{p}=\lambda^{-4}\left\|E_{[0,1]}^{\gamma}\widetilde{f}\right\|_{L^{p}(\mathcal{L}(B_{K}^{2}))}^{p}=\lambda^{-4}\sum_{B_{\lambda^{3}K}^{2}\subseteq \mathcal{L}(B_{K}^{2})}\left\|E_{[0,1]}^{\gamma}\widetilde{f}\right\|_{L^{p}(B_{\lambda^{3}K}^{2})}^{p}
			\end{equation*}
			\begin{equation*}
				\lesssim_{\varepsilon}\lambda^{-4}K^{p\varepsilon}\sum_{B_{\lambda^{3}K}^{2}\subseteq \mathcal{L}(B_{K}^{2})}\left(\sum_{\widetilde{I}:\lambda^{-3/2}K^{-1/2}-\text{interval}}\left\|E_{\widetilde{I}}^{\gamma}\widetilde{f}\right\|_{L^{p}(B_{\lambda^{3}K}^{2})}^{2}\right)^{p/2}
			\end{equation*}
			\begin{equation*}
				=\lambda^{-4}K^{p\varepsilon}\sum_{B_{\lambda^{3}K}^{2}\subseteq \mathcal{L}(B_{K}^{2})}\left\|\left\{\left\|E_{\widetilde{I}}^{\gamma}\widetilde{f}\right\|_{L^{p}(B_{\lambda^{3}K}^{2})}^{p}\right\}_{\widetilde{I}}\right\|_{\ell^{\frac{2}{p}}}
			\end{equation*}
			\begin{equation*}
				\le\lambda^{-4}K^{p\varepsilon}\left\|\left\{\sum_{B_{\lambda^{3}K}^{2}\subseteq \mathcal{L}(B_{K}^{2})}\left\|E_{\widetilde{I}}^{\gamma}\widetilde{f}\right\|_{L^{p}(B_{\lambda^{3}K}^{2})}^{p}\right\}_{\widetilde{I}}\right\|_{\ell^{\frac{2}{p}}}
			\end{equation*}
			\begin{equation*}
				\lesssim_{\varepsilon}\lambda^{-4}K^{p\varepsilon}\left(\sum_{\widetilde{I}:\lambda^{-3/2}K^{-1/2}-\text{interval}}\left\|E_{\widetilde{I}}^{\gamma}\widetilde{f}\right\|_{L^{p}(\mathcal{L}(B_{K}^{2}))}^{2}\right)^{p/2}
			\end{equation*}
			\begin{equation*}
				=K^{p\varepsilon} \left(\sum_{I:\lambda^{-1/2}K^{-1/2}-\text{interval} }\|E_{I}f\|_{L^{p}(B_{K}^{2})}^{2}\right)^{p/2}.
			\end{equation*}
		\end{proof}
		 As a direct consequence, we can derive a decoupling inequality on $\Omega_{\lambda}$.
		 \begin{coro}\label{3.2}
		 		For $2\le p\le 6$, $\Gamma_{\lambda}:=\lbrace(\xi_{1},\xi_{2},\xi_{1}^{3}+\xi_{2}^{3}):(\xi_{1},\xi_{2})\in \Omega_{\lambda}\rbrace$, $\supp\;   \widetilde{G}\subseteq \mathcal{N}_{K^{-1}}(\Gamma_{\lambda})$, we have 
		 		\begin{equation*}
		 			\|G\|_{L^{p}(B_{K}^{3})}\lesssim_{\varepsilon} K^{\varepsilon} \left(\sum_{\tau}\|G_{\tau}\|_{L^{p}(B_{K}^{3})}^{2}\right)^{1/2}, \quad \widetilde{G_{\tau}}=\chi_{\tau}\widetilde{G},
		 		\end{equation*}
		 		where the sum is over all $\tau: \lambda^{-1/2}K^{-1/2}\times K^{-1/3}\times K^{-1}$--slabs contained in $\mathcal{N}_{K^{-1}}(\Gamma_{\lambda})$.
		  \end{coro}
		\begin{proof}
			We first denote $G_{x_{2}}(x_{1},x_{3}):=G(x_{1},x_{2},x_{3})$. Then, 
			\begin{equation*}
				\supp\; \widehat{G_{x_{2}}}\subseteq \mathcal{N}_{K^{-1}}(\gamma_{\lambda}).
			\end{equation*}
			Applying Lemma \ref{3.3} and the (inverse) Minkowski's inequality, we obtain
			\begin{equation*}
				\|G\|_{L^{p}(B_{K}^{3})}^{p}=\int_{-K}^{K}\|G_{x_{2}}\|_{L^{p}(B_{K}^{2})}^{p} dx_{2}\lesssim_{\varepsilon} K^{p\varepsilon}\int_{-K}^{K} \left(\sum_{\theta}\|(G_{x_{2}})_{\theta}\|_{L^{p}(B_{K}^{2})}^{2}\right)^{p/2} dx_{2}
			\end{equation*}
			\begin{equation*}
				=K^{p\varepsilon}\int_{-K}^{K} \left(\sum_{\tau}\|(G_{\tau})_{x_{2}}\|_{L^{p}(B_{K}^{2})}^{2}\right)^{p/2} dx_{2}
			\end{equation*}
			\begin{equation*}
				=K^{p\varepsilon}\int_{-K}^{K}\left\|\left \lbrace\|(G_{\tau})_{x_{2}}\|_{L^{p}(B_{K}^{2})}^{p}\right\rbrace_{\tau}\right\|_{\ell^{2/p}}dx_{2}
			\end{equation*}
			\begin{equation*}
				\le K^{p\varepsilon} \left\|\left \lbrace\int_{-K}^{K}\|(G_{\tau})_{x_{2}}\|_{L^{p}(B_{K}^{2})}^{p}dx_{2}\right\rbrace_{\tau}\right\|_{\ell^{2/p}}	=K^{p\varepsilon}\left(\sum_{\tau}\|G_{\tau}\|_{L^{p}(B_{K}^{3})}^{2}\right)^{p/2}.
			\end{equation*}
			
		\end{proof}
		Now we are ready to establish the restriction estimate on $\Omega_{\lambda}$, i.e.,
		\begin{equation*}
			\left\|E g_{\Omega_{\lambda}}\right\|_{L^{22/7}(B_{R}^{3})}\lesssim_{\varepsilon} C(K)R^{\varepsilon}\|g\|_{L^{\infty}([0,1]^{2})}.
		\end{equation*}
		
			Applying the equivalent form of Corollary \ref{3.2}, we have 
			\begin{equation*}
				\left\|Eg_{\Omega_{\lambda}}\right\|_{L^{22/7}(B_{K}^{3})}\lesssim_{\varepsilon}\left(\sum_{\tau}\left\|Eg_{\tau\cap\Omega_{\lambda}}\right\|_{L^{22/7}(B_{K}^{3})}^{2}\right)^{1/2},
			\end{equation*}
			where the sum is over all $\tau:$ $\lambda^{-1/2}K^{-1/2}\times K^{-1/3}$--slab in $\Omega_{\lambda}$.
			
			Then, using Lemma \ref{parallel} and summing over all $B_{K}^{3}\subseteq B_{R}^{3}$, we obtain 
			\begin{equation}\label{a}
				\left\|Eg_{\Omega_{\lambda}}\right\|_{L^{22/7}(B_{R}^{3})}\lesssim_{\varepsilon}\left(\sum_{\tau}\left\|Eg_{\tau\cap\Omega_{\lambda}}\right\|_{L^{22/7}(B_{R}^{3})}^{2}\right)^{1/2}.
			\end{equation}
			Now for each $\tau\cap\Omega_{\lambda}$, we adopt the following change of variables.
			\begin{equation*}
				\left\{
				\begin{aligned}
					& \xi_{1}:=\lambda+\lambda^{-1/2}K^{-1/2}\eta_{1},  \\
					& \xi_{2}:=K^{-1/3} \eta_{2},
				\end{aligned}
				\right.
			\end{equation*}
			where $(\eta_{1},\eta_{2})\in [0,1]^{2}$. Then, we obtain
			\begin{equation}\label{x}
				\Big|Eg_{\tau\cap\Omega_{\lambda}}(x)\Big|=\Big|E_{[0,1]^{2}}^{\Gamma}\widetilde{g}(\widetilde{x})\Big|,
			\end{equation}
			where $\widetilde{x}=(\widetilde{x_{1}},\widetilde{x_{2}},\widetilde{x_{3}})$, $\widetilde{g}$ and $\Gamma$ satisfy
			\begin{equation}\label{z}
				\left\{
				\begin{aligned}
					& \widetilde{x_{1}}:=\lambda^{-1/2}K^{-1/2}x_{1}+3\lambda^{3/2}K^{-1/2}x_{3}, \; \widetilde{x_{2}}:=K^{-1/3} x_{2},\;  \widetilde{x_{3}}:=K^{-1}x_{3}, \\
					& \widetilde{g}(\eta_{1},\eta_{2}):=\lambda^{-1/2}K^{-5/6}g(\lambda+\lambda^{-1/2}K^{-1/2}\eta_{1},K^{-1/3} \eta_{2}),\\
					& \Gamma(\eta_{1},\eta_{2}):=3\eta_{1}^{2}+\eta_{2}^{3}+\lambda^{-3/2}K^{-1/2}\eta_{1}^{3}.
				\end{aligned}
				\right.
			\end{equation}
			However, we can not directly apply Theorem \ref{WW} to $E_{[0,1]^{2}}^{\Gamma}$ there, since the term $\eta_{2}^{3}$ degenerates at $\eta_{2}=0$. 
			
			A key observation is that we can further decompose $[0,1]^{2}$ into two parts, one away from $\eta_{2}=0$ and one near $\eta_{2}=0$ (see Figure 5). Then, we can obtain an estimate for $E_{[0,1]^{2}}^{\Gamma}$ as follows. The idea is inspired by \cite{64}.
			
				\begin{figure}
				\centering
				\includegraphics[width=0.7\linewidth]{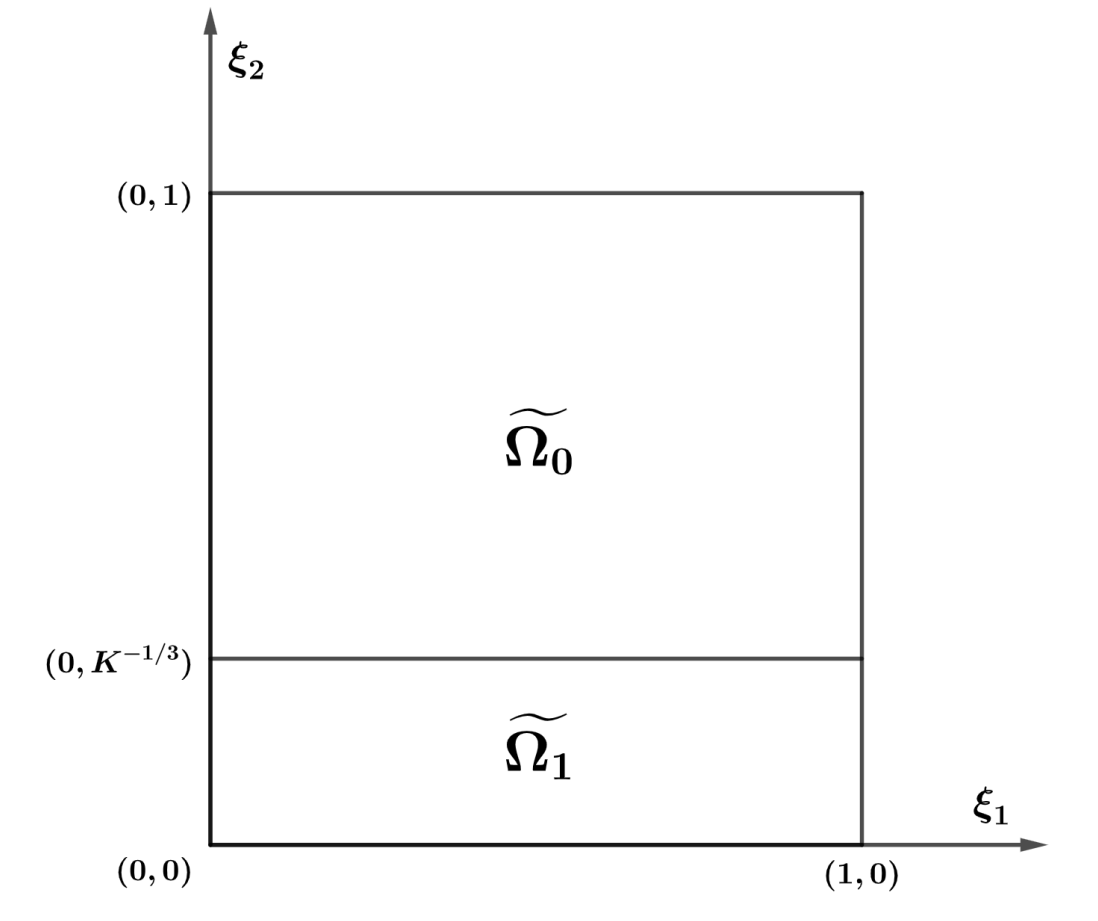}
				\caption{}
				\label{fig:omega2}
			\end{figure}
			
			\begin{prop}\label{23}
				For any $\varepsilon>0$, there exists $\widetilde{C}(\varepsilon)>0$, such that 
				\begin{equation*}
					\|E_{[0,1]^{2}}^{\Gamma}\widetilde{f}\|_{L^{22/7}(B_{R}^{3})}\le \widetilde{C}(\varepsilon)R^{\varepsilon}\|\widetilde{f}\|_{L^{\infty}([0,1]^{2})}, \quad \forall \widetilde{f}\in L^{\infty}([0,1]^{2}).
				\end{equation*}
			\end{prop}
			\begin{proof}
				In fact, we will establish this restriction estimate in a more general setting.
				
				Let $A(R)$ be the optimal constant such that
				\begin{equation*}
					\|E_{S}\widetilde{f}\|_{L^{22/7}(B_{R}^{3})}\le A(R)\|\widetilde{f}\|_{L^{\infty}([0,1]^{2})}, \quad \forall \widetilde{f}\in L^{\infty}([0,1]^{2}),
				\end{equation*}
				holds for all surfaces $S:=\lbrace(\xi_{1},\xi_{2},\phi(\xi_{1})+\phi(\xi_{2})):(\xi_{1},\xi_{2})\in[0,1]^{2}\rbrace$, where $\phi_{1}$ is of finite type $2$ and $\phi_{2}$ is of finite type $3$.
				
				We decompose the region $[0,1]^{2}$ into $\widetilde{\Omega_{0}}\cup\widetilde{\Omega_{1}}$ (see Figure 5).
				
				 On $\widetilde{\Omega_{0}}$, surface $S$ is non-degenerate, which directly implies that for any $\varepsilon>0$, there exists $A_{0}(\varepsilon)$ such that
				\begin{equation}\label{bn}
					\|E_{S}\widetilde{f}_{\widetilde{\Omega_{0}}}\|_{L^{22/7}(B_{R}^{3})}\le  A_{0}(\varepsilon)R^{\varepsilon}\|\widetilde{f}\|_{L^{\infty}([0,1]^{2})}, \quad \forall \widetilde{f}\in L^{\infty}([0,1]^{2}).
				\end{equation}
				On $\widetilde{\Omega_{1}}$, we can proceed with a similar argument as in Lemma \ref{3.3} and Corollary \ref{3.2}. Then, we obtain the following $\ell^{2}$-decoupling inequality
				\begin{equation}\label{nm}
					\|E_{S}\widetilde{f}_{\widetilde{\Omega_{1}}}\|_{L^{22/7}(B_{R}^{3})}\lesssim_{\delta}K^{\delta}\left(\sum_{\widetilde{\tau}}\|E_{S}\widetilde{f}_{\widetilde{\tau}}\|_{L^{22/7}(B_{R}^{3})}^{2}\right)^{1/2}, \quad \forall \delta>0,
				\end{equation}
				where the sum is over all $\widetilde{\tau}:$ $K^{-1/2}\times K^{-1/3}$-slabs contained in $\widetilde{\Omega_{1}}$ (see Figure 6).
				
				\begin{figure}
					\centering
					\includegraphics[width=0.7\linewidth]{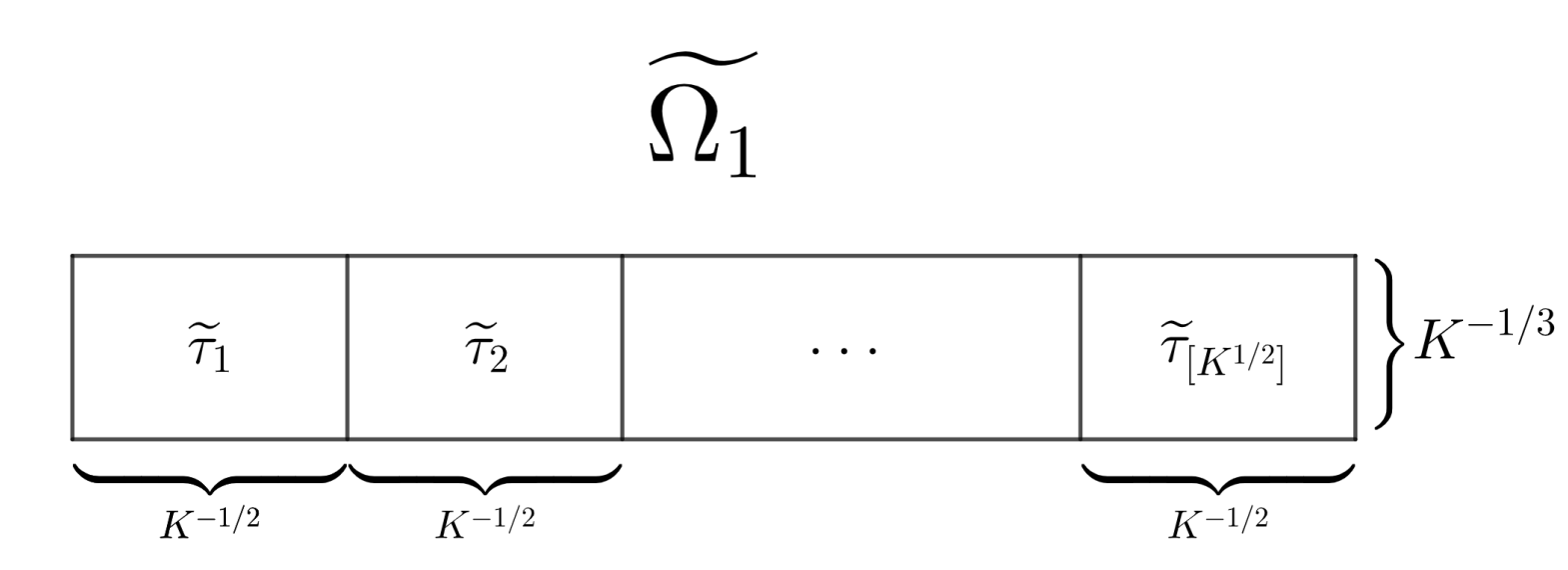}
					\caption{}
					\label{fig:decomposition}
				\end{figure}
				
				For simplicity, we assume $\widetilde{\tau}$ is the first slab $\widetilde{\tau}_{1}$ in Figure 6, i.e., $\widetilde{\tau}=[0,K^{-1/2}]\times[0,K^{-1/3}]$. Taking the following change of variables,
				\begin{equation*}
					\left\{
					\begin{aligned}
						& \overline{\eta}_{1}:=K^{1/2}\eta_{1},  \\
						& \overline{\eta}_{2}:=K^{1/3}\eta_{2},
					\end{aligned}
					\right.
				\end{equation*}
				we have 
				\begin{equation*}
					|E_{S}\widetilde{f}_{\widetilde{\tau}}(\widetilde{x})|=\Big|\int_{[0,1]^{2}}e^{i(\overline{x}_{1}\overline{\eta}_{1}+\overline{x}_{2}\overline{\eta}_{2}+\overline{x}_{3}\overline{\Gamma}(\overline{\eta}_{1},\overline{\eta}_{2}))}\overline{f}(\overline{\eta}_{1},\overline{\eta}_{2})d\overline{\eta}_{1}d\overline{\eta}_{2}\Big|:=|\widetilde{E}_{[0,1]^{2}}^{\overline{\Gamma}}\overline{f}(\overline{x})|,
				\end{equation*}
				where $\overline{x}=(\overline{x}_{1},\overline{x}_{2},\overline{x}_{3})$, $\overline{\Gamma}(\overline{\eta}_{1},\overline{\eta}_{2})=\overline{\phi}_{1}(\overline{\eta}_{1})+\overline{\phi}_{2}(\overline{\eta}_{2})$, $\overline{f}$, satisfying
				\begin{equation*}
					\left\{
					\begin{aligned}
						& \overline{x}_{1}:=K^{-1/2}\widetilde{x}_{1},\; \overline{x}_{2}:=K^{-1/3}\widetilde{x}_{2},\; \overline{x}_{3}:=K^{-1}\widetilde{x}_{3},  \\
						& \overline{f}(\overline{\eta}_{1},\overline{\eta}_{2}):=K^{-5/6}\widetilde{f}(K^{-1/2}\overline{\eta}_{1},K^{-1/3}\overline{\eta}_{2}), \\
						& \overline{\phi}_{1}(\overline{\eta}_{1}):=K\phi_{1}(K^{-1/2}\overline{\eta}_{1}),\; \overline{\phi}_{2}(\overline{\eta}_{2}):=K\phi_{2}(K^{-1/3}\overline{\eta}_{2}).
					\end{aligned}
					\right.
				\end{equation*}
				Note that the phase function $\overline{\phi}_{1}$ is still of finite type $2$, and $\overline{\phi}_{2}$ is of finite type $3$. Now we can obtain
				\begin{equation*}
					\|E_{S}\widetilde{f}_{\widetilde{\tau}}\|_{L^{22/7}(B_{R}^{3})}\lesssim K^{7/12}\|\widetilde{E}_{[0,1]^{2}}^{\overline{\Gamma}}\overline{f}\|_{L^{22/7}(B_{RK^{-1/3}}^{3})}\lesssim K^{-1/4}A\left(\frac{R}{K^{1/3}}\right)\|\widetilde{f}\|_{L^{\infty}([0,1]^{2})}.
				\end{equation*}
				Substituting it into (\ref{nm}) and taking $\delta=\frac{\varepsilon}{10}$, we obtain 
				\begin{equation*}
					\|E_{S}\widetilde{f}_{\widetilde{\Omega_{1}}}\|_{L^{22/7}(B_{R}^{3})}\le\widetilde{A}_{0}(\varepsilon) K^{\varepsilon/10}A\left(\frac{R}{K^{1/3}}\right)\|\widetilde{f}\|_{L^{\infty}([0,1]^{2})},
				\end{equation*}
				for some constant $\widetilde{A}_{0}(\varepsilon)>0$.
				
				Combining this with the estimate (\ref{bn}), we obtain
				\begin{equation*}
					A(R)\le A_{0}(\varepsilon)R^{\varepsilon}+\widetilde{A}_{0}(\varepsilon)K^{\varepsilon/10} A\left(\frac{R}{K^{1/3}}\right).
				\end{equation*}
				We then immediately derive the desired bound $A(R)\lesssim_{\varepsilon}R^{\varepsilon}$.
			\end{proof}
			\begin{rem}
				In this proof, we see that the exponent $22/7$ is actually optimal for our rescaling argument. In fact, if we replace $22/7$ with a general $p$ in Proposition \ref{23} and assume the $L^{\infty}\to L^{p}$ estimate for the perturbed paraboloid, then we obtain
				\begin{equation*}
					\|E_{S}\widetilde{f}_{\widetilde{\Omega_{1}}}\|_{L^{p}(B_{R}^{3})}\le\widetilde{A}_{0}(\varepsilon) K^{\frac{11}{6p}-\frac{7}{12}+\frac{\varepsilon}{10}}A\left(\frac{R}{K^{1/3}}\right)\|\widetilde{f}\|_{L^{\infty}([0,1]^{2})},
				\end{equation*}
				which leads to the recursive relation as follows
				\begin{equation*}
					A(R)\le A_{0}(\varepsilon)R^{\varepsilon}+\widetilde{A}_{0}(\varepsilon)K^{\frac{11}{6p}-\frac{7}{12}+\frac{\varepsilon}{10}} A\left(\frac{R}{K^{1/3}}\right).
				\end{equation*}
				To close the induction, we need $\frac{11}{6p}-\frac{7}{12}\le 0$, i.e., $p\ge 22/7$.
			\end{rem}
				\begin{rem}
				In \cite{15}, Li, Miao, and Zheng took another way to deal with the degeneracy of $\Gamma$ on $\eta_{2}=0$. In fact, they further decomposed the region $\Omega_{\lambda}$ (in their case, $\Omega_{\lambda}:=[\lambda,2\lambda]\times [0,K^{-1/4}]$) in the $y$--direction as follows.
				\begin{equation*}
					\Omega_{\lambda}:=\Omega_{\lambda,0}\cup\left(\bigcup_{\sigma}\Omega_{\lambda,\sigma}\right), \quad \sigma=K^{j/4}R^{-1/4}, \; 1\le j\le [\log_{K}R]-1,
				\end{equation*}
				where $\Omega_{\lambda,0}:=[\lambda,2\lambda]\times[0,R^{-1/4}],\; \Omega_{\lambda,\sigma}:=[\lambda,2\lambda]\times[K^{-1/4}\sigma,\sigma].$
				
				For $\Omega_{\lambda,\sigma}$, it is a region away from $\eta_{2}=0$, then there is no difficulty in applying the restriction estimate for the perturbed paraboloid to $\Gamma(\eta_{1},\eta_{2})$.
				
				Using the reverse square function estimate and wave packet decomposition, they successfully reduced the estimate for $\Omega_{\lambda,0}$ to a Kakeya type estimate.
				
				Their method also applies to our setting (surfaces of finite type 3), though the argument is less streamlined than the one above. The analogue of Proposition \ref{23}, however, is not very effective in their setting (surfaces of finite type 4). In fact, the method in Proposition \ref{23} will lead to
				\begin{equation*}
					A(R)\le A_{0}R^{\varepsilon}+\widetilde{A_{0}}(\varepsilon)K^{\frac{7}{4p}-\frac{1}{2}+\frac{\varepsilon}{10}}A\left(\frac{R}{K^{1/4}}\right),
				\end{equation*}
				where $A(R)$ is the optimal constant such that
				\begin{equation*}
					\|E_{S}\widetilde{f}\|_{L^{p}(B_{R}^{3})}\le A(R)\|\widetilde{f}\|_{L^{\infty}([0,1]^{2})}, \quad \forall \widetilde{f}\in L^{\infty}([0,1]^{2}),
				\end{equation*}
				holds for all surfaces $S:=\lbrace(\xi_{1},\xi_{2},\phi(\xi_{1})+\phi(\xi_{2})):(\xi_{1},\xi_{2})\in[0,1]^{2}\rbrace$, where $\phi_{1}$ is of finite type $2$ and $\phi_{2}$ is of finite type $4$.
				
				Now, we are forced to require $\frac{7}{4p}-\frac{1}{2}\le 0$, i.e., $p\ge 3.5$, which is worse than the exponent obtained in \cite{15}. Thus, although the method in Proposition \ref{23} is considerably simpler, it has certain limitations.
				
			\end{rem}
			\vspace{10pt}
			Back to the estimate for $Eg_{\Omega_{\lambda}}$, we now have 
			\begin{equation*}
				\|Eg_{\tau\cap\Omega_{\lambda}}\|_{L^{p}(B_{R}^{3})}\lesssim K^{7/12}\|E_{[0,1]^{2}}^{\Gamma}\widetilde{g}\|_{L^{p}(B_{RK^{-1/3}}^{3})}\lesssim_{\varepsilon}K^{-1/4}\left(\frac{R}{K^{1/3}}\right)^{\varepsilon}\|g\|_{L^{\infty}([0,1]^{2})}.
			\end{equation*}
			Then, we can plug this estimate into (\ref{a}) and obtain that
			\begin{equation*}
			    \|Eg_{\Omega_{\lambda}}\|_{L^{22/7}(B_{R}^{3})}\lesssim_{\varepsilon}C(K)R^{\varepsilon}\|g\|_{L^{\infty}([0,1]^{2})}, \quad \forall \varepsilon>0,
			\end{equation*}
			which finally completes the whole proof.

		\begin{rem}\label{c}
			So far, we have proved the prototypical case of Theorem \ref{main}, i.e., restriction estimate for $S:=\lbrace(\xi_{1},\xi_{2},\xi_{1}^{3}+\xi_{2}^{3}):(\xi_{1},\xi_{2})\in [0,1]^{2}\rbrace$. The proof for the more general case $S:=\lbrace(\xi_{1},\xi_{2},\varphi_{1}(\xi_{1})+\varphi_{2}(\xi_{2})):(\xi_{1},\xi_{2})\in [0,1]^{2}\rbrace$ is similar, and we only need to check some important details as follows.
			
				From the definition of the finite-type 3 functions $\varphi_{1}, \varphi_{2}$, we obtain the Taylor expansions:
				\begin{equation*}
					\varphi_{1}(\xi_{1})=a_{3}\xi_{1}^{3}+\sum_{k=4}^{\infty}a_{k}\xi_{1}^{k},\quad \varphi_{2}(\xi_{2})=b_{3}\xi_{2}^{3}+\sum_{k=4}^{\infty}b_{k}\xi_{2}^{k},
				\end{equation*}
				where $a_{3}\sim1\sim b_{3}$, and $|a_{k}|,|b_{k}|\ll \frac{1}{k!}$, $\forall k\ge4$.
				\begin{itemize}
				
				\item For rescaling (e.g., Lemma \ref{3.3}):
				
				We can still change the variables $t:=\lambda u+\lambda$ and have
				\begin{equation*}
					\Big|E_{[\lambda,2\lambda]}f(y)\Big|=\Big|\int_{\lambda}^{2\lambda}f(t)e^{i(y_{1}t+y_{2}\varphi_{1}(t))}dt\Big|	=\Big|\int_{0}^{1}\widetilde{f}(u)e^{i(\widetilde{y_{1}}u+\widetilde{y}_{2}\gamma(u))}du\Big|\triangleq\Big|E_{[0,1]}^{\gamma}\widetilde{f}(\widetilde{y})\Big|,
				\end{equation*}
				where we denote $\widetilde{f}(u)=\lambda f(\lambda u+\lambda)$. Now the new phase function $\gamma(u)$ and the linear mapping $\mathcal{L}$ are given by
				\begin{equation*}
					\left\{
					\begin{aligned}
						& \gamma(u)=R_{2}u^{2}+R_{3}u^{3}+\sum_{k=4}^{\infty}R_{k}u^{k},  \\
						& 	\mathcal{L}(y)=(\widetilde{y_{1}},\widetilde{y_{2}})=(\lambda y_{1}+\lambda^{3}R_{1}y_{2}, \lambda^{3}y_{2}),
					\end{aligned}
					\right.
				\end{equation*}
				where we define
				\begin{equation*}
					R_{m}:=\left\{
					\begin{aligned}
						& a_{3}\binom{3}{m}+\sum_{k=4}^{\infty}\lambda^{k-3}a_{k}\binom{k}{m},  \quad 1\le m\le 3,\\
						& \sum_{k=4}^{\infty}\lambda^{k-3}a_{k}\binom{k}{m},  \quad m\ge 3.
					\end{aligned}
					\right.
				\end{equation*}
				Then, from the following simple observation (uniform for all finite-type functions),
				\begin{equation*}
					\Big|R_{m}-a_{3}\binom{3}{m}\Big|\ll 1, \quad \forall1\le m\le 3; \quad  |R_{m}|\ll 1, \quad \forall m\ge 3,
				\end{equation*}
				we can still apply the $\ell^{2}$--decoupling to prove Lemma \ref{3.3}.
			\item For induction:
			
			We need to slightly modify the definition of the optimal constant $Q(R)$ in (\ref{n}) as follows.
			
		Let $Q(R)$ be the optimal constant such that 
		\begin{equation*}
			\|E_{S}g\|_{L^{22/7}(B_{R}^{3})}
			\le Q(R)\|g\|_{L^{\infty}([0,1]^{2})}, \quad \forall g\in L^{\infty}([0,1]^{2}),
		\end{equation*}
			holds for all surfaces $S$ of finite type $3$.

		Then, we can still perform the rescaling for the estimate on $\Omega_{3}$ as follows.
		\begin{equation*}
			\Big|E_{S} g_{\Omega_{3}}(y)\Big|=\Big|\int_{[0,K^{-1/3}]^{2}}e^{i(y_{1}\xi_{1}+y_{2}\xi_{2}+y_{3}(\varphi_{1}(\xi_{1})+\varphi_{2}(\xi_{2})))}g(\xi_{1},\xi_{2})d\xi_{1}d\xi_{2}\Big|
		\end{equation*}
		\begin{equation*}
			=\Big|\int_{[0,1]^{2}}e^{i(\widetilde{y}_{1}\xi_{1}+\widetilde{y}_{2}\xi_{2}+\widetilde{y}_{3}(\widetilde{\varphi}_{1}(\eta_{1})+\widetilde{\varphi}_{2}(\eta_{2})))}\widetilde{g}(\eta_{1},\eta_{2})d\eta_{1}d\eta_{2}\Big|\triangleq\Big|\widetilde{E}_{[0,1]^{2}}\widetilde{g}(\widetilde{y})\Big|,
		\end{equation*}
		where $\widetilde{y}=(\widetilde{y_{1}},\widetilde{y_{2}},\widetilde{y_{3}})$,
		\begin{equation*}
		    \left\{
			\begin{aligned}
				& \widetilde{y}_{1}:=K^{-1/3}y_{1},\; \widetilde{y}_{2}:=K^{-1/3}y_{2},\; \widetilde{y}_{3}:=K^{-1}y_{3},\\
				& \widetilde{g}(\eta_{1},\eta_{2}):=K^{-2/3}g(K^{-1/3}\eta_{1},K^{-1/3}\eta_{2}),\\
				& \widetilde{\varphi}_{1}(\eta_{1}):=K\varphi_{1}(K^{-1/3}\eta_{1}),\; \widetilde{\varphi}_{2}(\eta_{2}):=K\varphi_{2}(K^{-1/3}\eta_{2}).
			\end{aligned}
			\right.
		\end{equation*}
			Then, a simple calculation shows that $\widetilde{\varphi}_{1}$ and $\widetilde{\varphi}_{2}$ are still functions of finite type $3$. Therefore, we can still obtain
			\begin{equation*}
				\|E_{S}g_{\Omega_{3}}\|_{L^{22/7}(B_{R}^{3})}\lesssim K^{35/66}\|\widetilde{E}_{[0,1]^{2}}\widetilde{g}\|_{L^{22/7}(B_{RK^{-1/3}}^{3})}
			\end{equation*}
			\begin{equation*}
				\lesssim K^{35/66}Q\left(\frac{R}{K^{1/3}}\right)\|\widetilde{g}\|_{L^{\infty}([0,1]^{2})}=K^{-3/22}Q\left(\frac{R}{K^{1/3}}\right)\|g\|_{L^{\infty}([0,1]^{2})}.
			\end{equation*}
			\end{itemize}
			
		\end{rem}
		\vspace{1pt}
		\begin{rem}
		The proof for the surface $
		S := \{(\xi_{1}, \xi_{2}, \xi_{1}^{3} + \xi_{2}^{3}) : (\xi_{1}, \xi_{2}) \in [0,1]^{2}\}$
		also applies to $
		S := \{(\xi_{1}, \xi_{2}, \xi_{1}^{3} - \xi_{2}^{3}) : (\xi_{1}, \xi_{2}) \in [0,1]^{2}\}.$
		Moreover, the condition $
		\frac{11}{2p} + \frac{3}{4} \le \frac{5}{2q'}$
		arises naturally from the same calculation after replacing \(22/7\) and \(\infty\) with \(p\) and \(q\), respectively.
			
			We only need to check the rescaling in (\ref{x}) and (\ref{z}). The only difference is that we now have 
			\begin{equation*}
				\left\{
				\begin{aligned}
					& \widetilde{x_{1}}:=\lambda^{-1/2}\sigma^{3/2}x_{1}+3\lambda^{3/2}\sigma^{3/2}x_{3}, \; \widetilde{x_{2}}:=\sigma x_{2},\;  \widetilde{x_{3}}:=\sigma^{3}x_{3}, \\
					& \widetilde{g}(\eta_{1},\eta_{2}):=\lambda^{-1/2}\sigma^{5/2}g(\xi_{1},\xi_{2}),\\
					& \Gamma(\eta_{1},\eta_{2}):=3\eta_{1}^{2}-\eta_{2}^{3}+\lambda^{-3/2}\sigma^{3/2}\eta_{1}^{3},
				\end{aligned}
				\right.
			\end{equation*}
			instead of $\Gamma(\eta_{1},\eta_{2}):=3\eta_{1}^{2}+\eta_{2}^{3}+\lambda^{-3/2}\sigma^{3/2}\eta_{1}^{3}$. In this case, we can establish a ``hyperbolic" analogue of Proposition \ref{23}, and there is no further difficulty.
			
			Similarly, we can derive Theorem \ref{main2} based on the modification in Remark \ref{c}.
		\end{rem}
		\vspace{1pt}
		\begin{rem}\label{32}
			In Proposition \ref{23}, we actually have proven the following restriction estimates, which are necessary for the application in the next section.
			
			\begin{figure}
				\centering
				\includegraphics[width=0.8\linewidth]{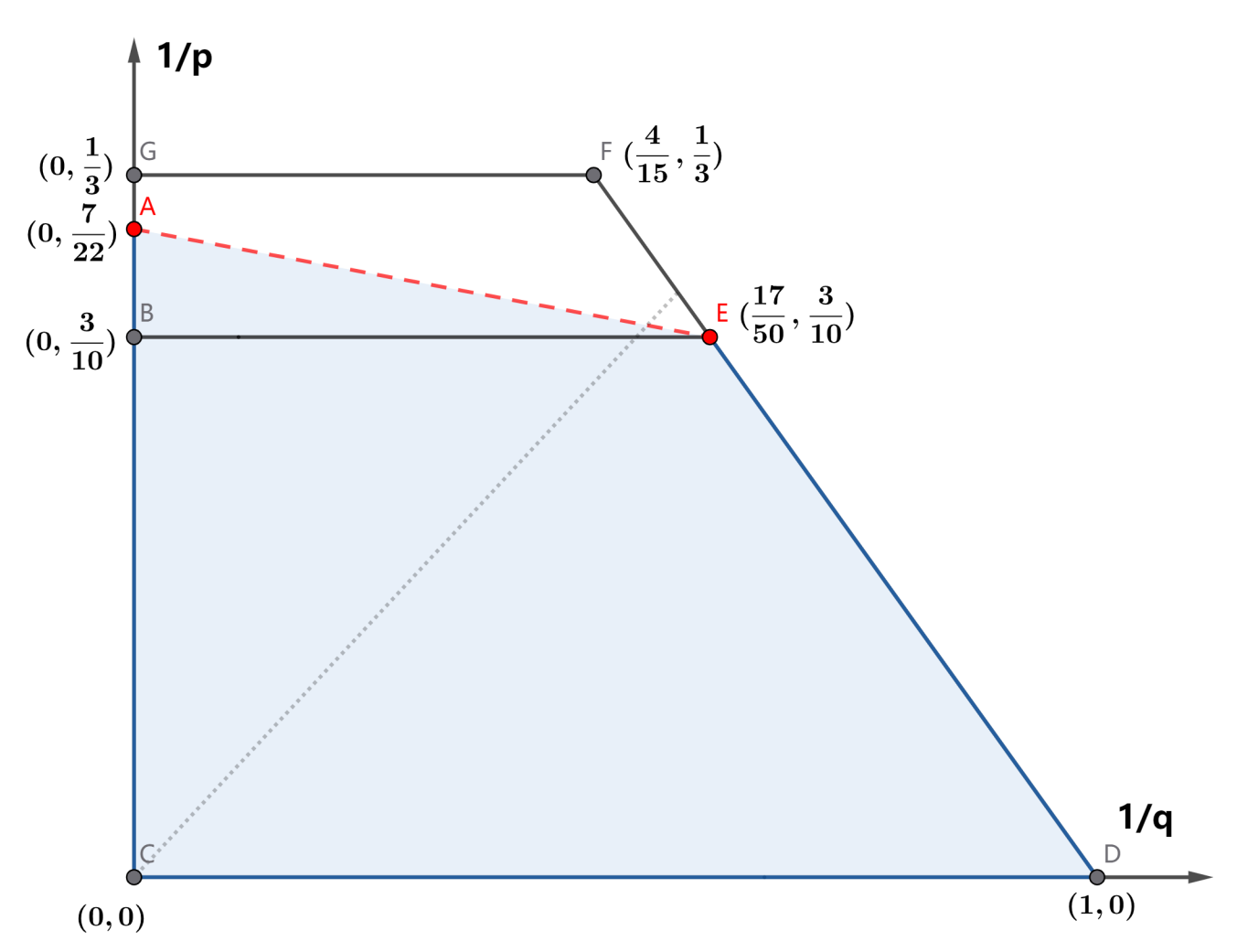}
				\caption{Range for $\phi_{1}(\xi_{1})+\phi_{2}(\xi_{2})$}
				\label{fig:restrictionrange3}
			\end{figure}
			
			\begin{figure}
				\centering
				\includegraphics[width=0.84\linewidth]{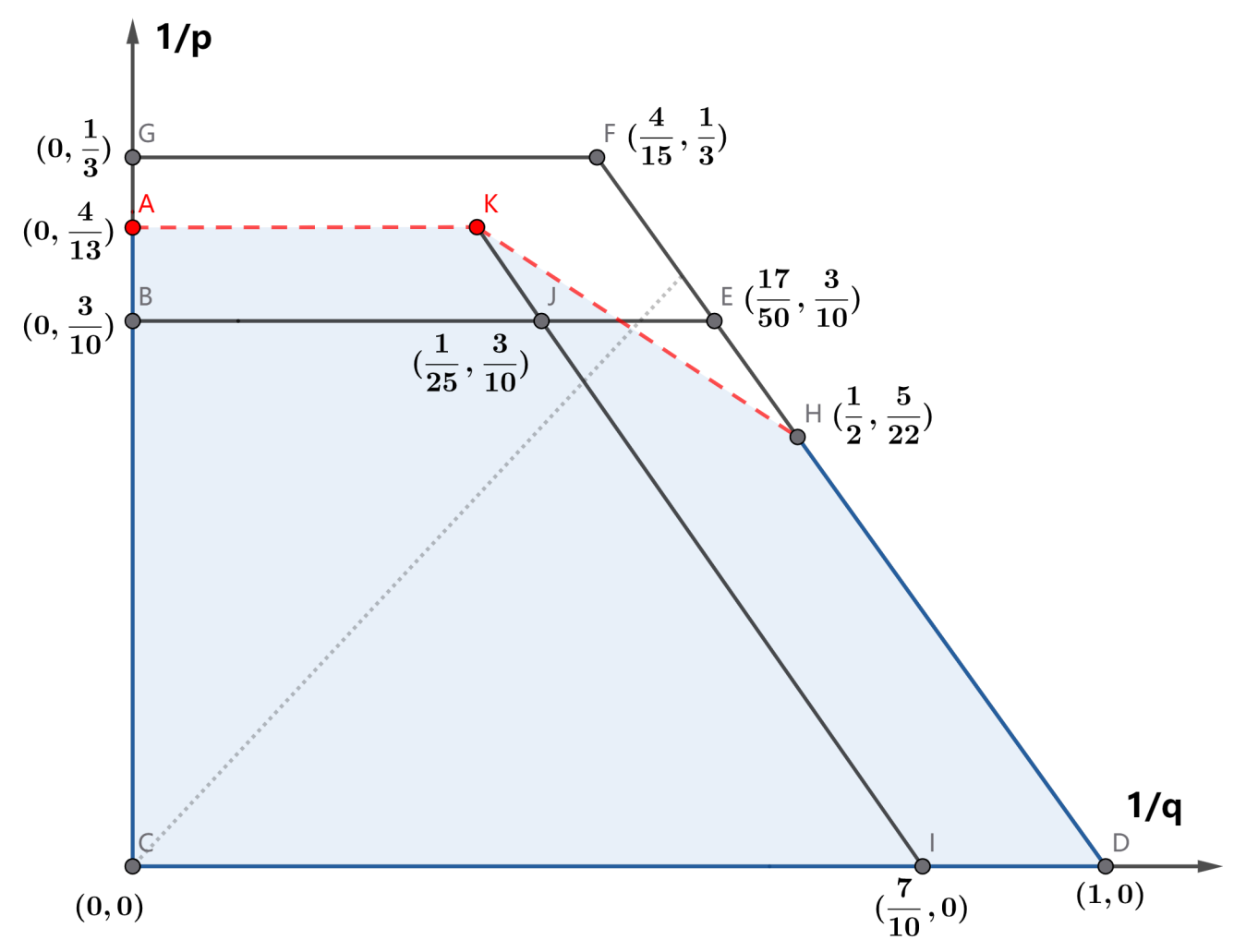}
				\caption{Range for $\phi_{1}(\xi_{1})-\phi_{2}(\xi_{2})$}
				\label{fig:restrictionrange4}
			\end{figure}
			\begin{itemize}
				\item Let $\varphi_{1}$ be of finite type $2$ and $\varphi_{2}$ be of finite type $3$, then for $S:=\lbrace(\xi_{1},\xi_{2},\phi_{1}(\xi_{1})+\phi_{2}(\xi_{2})):(\xi_{1},\xi_{2})\in [0,1]^{2}\rbrace$, we have the local restriction estimate
				\begin{equation*}
					\left\|E_{S} g\right\|_{L^{22/7}(B_{R}^{3})}\lesssim_{\varepsilon} R^{\varepsilon}\|g\|_{L^{\infty}([0,1]^{2})}, \quad \forall \varepsilon>0.
				\end{equation*}
				In the same spirit, for $S:=\lbrace(\xi_{1},\xi_{2},\phi_{1}(\xi_{1})-\phi_{2}(\xi_{2})):(\xi_{1},\xi_{2})\in [0,1]^{2}\rbrace$, $\frac{3}{10}\le \frac{1}{p}<\frac{4}{13}$, $\frac{11}{2p}+\frac{3}{4}\le \frac{5}{2q'}$, we have the local restriction estimate 
				\begin{equation*}
					\left\|E_{S} g\right\|_{L^{p}(B_{R}^{3})}\lesssim_{\varepsilon} R^{\varepsilon}\|g\|_{L^{q}([0,1]^{2})}, \quad \forall \varepsilon>0.
				\end{equation*}
				\item For $S:=\lbrace(\xi_{1},\xi_{2},\phi_{1}(\xi_{1})+\phi_{2}(\xi_{2})):(\xi_{1},\xi_{2})\in [0,1]^{2}\rbrace$, Buschenhenke, M\"{u}ller, and Vargas proved the case when ($\frac{1}{q},\frac{1}{p}$) lies in the trapezoid BCDE (excluding the segment $\overline{BE}$) in Figure 7. Now, by interpolation, we can establish the case in the quadrilateral ACDE (excluding the segments $\overline{AE}$). Moreover, the conjectured range for the prototypical case is the trapezoid GCDF (excluding the segment $\overline{GF}$).
				\vspace{5pt}
				\item For $S:=\lbrace(\xi_{1},\xi_{2},\phi_{1}(\xi_{1})-\phi_{2}(\xi_{2})):(\xi_{1},\xi_{2})\in [0,1]^{2}\rbrace$, we have the Tomas--Stein type estimate (the point H in Figure 8) instead. Then, still by interpolation, we can prove the case when ($\frac{1}{q},\frac{1}{p}$) lies in the pentagon ACDHK (excluding the segments $\overline{AK}$, $\overline{KH}$).
			\end{itemize}
		\end{rem}

		\vspace{5pt}
		For the rest of this section, we give a simple proof of the non--endpoint case in Theorem \ref{2D} using the same rescaling technique. 
		\begin{proof}[Proof of Theorem \ref{2D}]
		    For simplicity, we only deal with the prototypical case, i.e., $\gamma:=\lbrace(t,t^{3}):t\in [-1,1]\rbrace$. 
		    	
		    Let $Q(R)$ be the optimal constant such that 
		    	\[
		    	\|E_{\gamma}g\|_{L^{p}(B_{R}^{2})}
		    	\le Q(R)\|g\|_{L^{q}([-1,1])}, \quad \forall g\in L^{q}([-1,1]).
		    	\]
		    	
		    Next, we can dyadically decompose $[-1,1]$ into the sets $\Omega_{\lambda}^{\pm}$.
		    \begin{equation*}
		    	[-1,1]:=\Omega_{0}\cup\left(\bigcup_{\lambda} \Omega_{\lambda}^{+}\right)\cup \left(\bigcup_{\lambda} \Omega_{\lambda}^{-}\right), \quad 1\le j\le \left[\frac{1}{3}\log_{2}(K)\right],
		    \end{equation*}
		    where $\Omega_{0}:=[-K^{-1/3},K^{-1/3}]$, $\Omega_{\lambda}^{+}:=[\lambda,2\lambda],\; \Omega_{\lambda}^{-}:=[-2\lambda,-\lambda]$, $\lambda=2^{j-1}K^{-1/3}$.
		    
		    On $\Omega_{0}$, we can use the induction on scale as before and obtain 
		    \begin{equation*}
		    	\|E g_{\Omega_{0}}\|_{L^{p}(B_{R}^{2})}\lesssim Q\left(\frac{R}{K^{1/3}}\right)K^{-\frac{1}{3q'}+\frac{4}{3p}}\|g\|_{L^{q}([-1,1])}.
		    \end{equation*}
		    On $\Omega_{\lambda}^{+}$, we first use rescaling to obtain
		    \begin{equation*}
		    	\Big| E g_{\Omega_{\lambda}^{+}}(y)\Big|=\Big|E_{[0,1]}^{\gamma}\widetilde{g}(\widetilde{y})\Big|,
		    \end{equation*}
		    where we still use the change of variables $t:=\lambda u+\lambda$, with
		    \begin{equation*}
		    	\left\{
		    	\begin{aligned}
		    		& \widetilde{g}(u)=\lambda g(\lambda u+\lambda),\; \gamma(u)=3u^{2}+u^{3},\\
		    		& \widetilde{y}=(\widetilde{y}_{1},\widetilde{y}_{2})=(\lambda y_{1}+3\lambda^{3}y_{2},\lambda^{3}y_{2}).\\
		    	\end{aligned}
		    	\right.
		    \end{equation*}
		    Then, applying the 2D restriction conjecture, we deduce
		    \begin{equation*}
		    	\|E g_{\Omega_{\lambda}^{+}}\|_{L^{p}(B_{R}^{2})}\lesssim_{\lambda,\varepsilon}R^{\varepsilon}\|g\|_{L^{q}([-1,1])}, \quad \forall \varepsilon>0.
		    \end{equation*}
		    Similarly, we can also derive
		    \begin{equation*}
		    	\|E g_{\Omega_{\lambda}^{-}}\|_{L^{p}(B_{R}^{2})}\lesssim_{\lambda,\varepsilon}R^{\varepsilon}\|g\|_{L^{q}([-1,1])}, \quad \forall \varepsilon>0.
		    \end{equation*}
		    Combining all the estimates above, we obtain
		    \begin{equation*}
		        Q(R)\lesssim Q\left(\frac{R}{K^{1/3}}\right)K^{-\frac{1}{3q'}+\frac{4}{3p}}+C(K)R^{\varepsilon}, \quad \forall \varepsilon>0,
		    \end{equation*}
		    which directly implies $Q(R)\lesssim_{\varepsilon}R^{\varepsilon}, \forall \varepsilon>0$. Note that the exponent $-\frac{1}{3q'}+\frac{4}{3p}\le 0$. 
		 \end{proof}
		 \vspace{10pt}
		 \section{Applications to dispersive equations}
		 In this section, we will apply our restriction estimate for surfaces of finite type $3$ to the analysis of the discrete nonlinear Schr\"{o}dinger equations (DNLS). 
		 
		 The discrete nonlinear equations (or, to be specific, nonlinear equations on lattices) have served as models for many important physical phenomena. For example, the non--exponential energy relaxation in solids \cite{31} and the self--focusing and collapse of Langmuir waves in plasmas \cite{34}. For further physical interest, we refer to \cite{32,33}. On the other hand, DNLS also has its own mathematical interest. In fact, one often discretizes PDEs for computer simulations or approximations. Then, establishing the theory of the continuum limit, i.e., whether the solutions of the discrete equations converge to the solutions of the original equations, is very essential. For more recent progress on continuum limit theories, see \cite{35,36,37,38}.
		 
		 Now we introduce the DNLS on the lattice $\mathbb{Z}^{d}$ as follows.
		 \begin{equation}\label{nonlinear}
		 	\left\{
		 	\begin{aligned}
		 		& i\partial_{t} u(x,t)+ \Delta_{disc} u(x,t) =\mu |u|^{\alpha-1}u,  \\
		 		& u(x,0) = f(x), \quad (x,t)\in\Z^d\times \R,
		 	\end{aligned}
		 	\right.
		 \end{equation}
		 where $\mu=\pm 1$, $\alpha>1$, and the discrete Laplacian $\Delta_{disc}$ is defined as
		 \begin{equation*}
		 	\Delta_{disc}u(x):=\sum_{i=1}^{d}\left(u(x+e_{j})+u(x-e_{j})-2u(x)\right), \quad x\in \mathbb{Z}^{d},
		 \end{equation*}
		 with the standard basis $\lbrace e_{j}\rbrace_{j=1}^{d}$ of $\mathbb{Z}^{d}$.

         To study the behavior of DNLS (\ref{nonlinear}) in the frequency space, we need to use the following discrete Fourier transform and its inverse.
         \begin{equation*}
         	\mathcal{F}_{disc}u(\xi):=\sum_{x\in \Z^{d}}u(x)e^{-ix\xi}, \quad \xi\in\T^{d},
         \end{equation*}
         \begin{equation*}
         	\mathcal{F}_{disc}^{-1}v(x):=\frac{1}{(2\pi)^{d}}\int_{\T^{d}}v(\xi)e^{ix\xi}d\xi, \quad x\in \Z^{d}.
         \end{equation*}
         Then, we can rewrite the discrete Laplacian $\Delta_{disc}$ as 
         \begin{equation}\label{propagator}
         	-\Delta_{disc} u=\mathcal{F}_{disc}^{-1}\left\{\left[\sum_{i=1}^{d}4\sin^{2}(\frac{\xi_{i}}{2})\right]\mathcal{F}_{disc}u\right\}\triangleq \mathcal{F}_{disc}^{-1}\left\{\omega(\xi)\mathcal{F}_{disc}u\right\}.
         \end{equation}
         To properly set up our discussion, we introduce the following $\ell^{p}$--spaces, for $1\le p\le \infty$.
         \begin{equation*}
         	\ell^{p}(\Z^{d})
         	:= \left\lbrace
         	u:\Z^{d}\to\C \;\Bigg|\;
         	\left( \sum_{x\in\Z^{d}} |u(x)|^{p} \right)^{\frac{1}{p}} < \infty
         	\right\rbrace,
         \end{equation*}
         with the $\ell^{p}$--norm 
         \begin{equation*}
         	\|u\|_{\ell^{p}(\mathbb{Z}^{d})}:=\left( \sum_{x\in\Z^{d}} |u(x)|^{p} \right)^{\frac{1}{p}}.
         \end{equation*}
         
         Recall that a basic tool in studying dispersive equations is the Strichartz estimate. For example, we consider the following inhomogeneous free Schr\"{o}dinger equations
         \begin{equation}\label{NLS}
         	\left\{
         	\begin{aligned}
         		& i\partial_{t} u(x,t)+ \Delta u(x,t) =F(x,t),  \\
         		& u(x,0) = f(x), \quad (x,t)\in\R^d\times \R.
         	\end{aligned}
         	\right.
         \end{equation}
         Strichartz (see \cite{18}) applied the works of Stein, Tomas and Segal (see \cite{19,20,21}) and was the first to establish 
         \begin{equation*}
         	\|u\|_{L_{x,t}^{\frac{2(d+2)}{d}}(\R^{d}\times\R)}\lesssim \|f\|_{L^{2}(\R^{d})}+\|F\|_{L_{x,t}^{\frac{2(d+2)}{d+4}}(\R^{d}\times\R)}.
         \end{equation*}
         Following those fundamental results, there has been a large body of work studying Strichartz estimates in more general spaces and applying them to the well--posedness theory of dispersive equations (see \cite{22,23}). Notably, Keel and Tao proposed a very general abstract framework for dispersive equations in \cite{24}. By utilizing this framework, we can easily derive the following Strichartz estimate for DNLS (see \cite{25}). The author also uses the following estimate to prove certain scattering properties of DNLS in \cite{26}. 
         
         \begin{prop}\label{Strichartz}
         	We call the exponent pair $(q,r)$ admissible, if the pair satisfies
         	\begin{equation*}
         		q,r\ge 2, \quad (q,r,d)\ne(2,\infty,3),\quad \frac{1}{q}+\frac{d}{3r}\le \frac{d}{6}.
         	\end{equation*}
         	Then, we have the following three types of Strichartz estimates 
         	\begin{equation}\label{app}
         		\|e^{it\Delta_{disc}}f\|_{L_{t}^{q}\ell_{x}^{r}(\mathbb{R}\times \Z^{d})}\lesssim \|f\|_{\ell^{2}(\Z^{d})},
         	\end{equation}
         	\begin{equation*}
         		\|\int_{\mathbb{R}}e^{-s\Delta_{disc}}F(\cdot,s)ds\|_{\ell^{2}(\Z^{d})}\lesssim \|F\|_{L_{t}^{q'}\ell_{x}^{r'}(\mathbb{R}\times \Z^{d})},
         	\end{equation*}
         	\begin{equation}\label{inhomo}
         		\|\int_{s<t}e^{i(t-s)\Delta_{disc}}F(\cdot,s)ds\|_{L_{t}^{q}\ell_{x}^{r}(\mathbb{R}\times \Z^{d})}\lesssim \|F\|_{L_{t}^{\widetilde{q}'}\ell_{x}^{\widetilde{r}'}(\mathbb{R}\times \Z^{d})},
         	\end{equation}
         	where $(q,r), (\widetilde{q},\widetilde{r})$ are admissible exponent pairs. 
         \end{prop}
         To establish the well--posedness theory with higher or lower integrability (compared with $\ell^{2}$), we may need some improvements in the following two directions. One direction is to establish an estimate (\ref{app}) on the Fourier--Lebesgue space $\widehat{L^{p}}$ using Fourier restriction. Another direction is enlarging the range of pairs $(q,r), (\widetilde{q},\widetilde{r})$ in estimate (\ref{inhomo}). 
         
         In this paper, we follow the former direction and introduce the Fourier--Lebesgue space $\widehat{L^{p}}$ as follows.
         
         \begin{defi}
         	For $1\le p \le \infty$, the Fourier--Lebesgue space $\widehat{L^{p}}(\mathbb{Z}^{d})$ is defined as 
         	\begin{equation*}
         		\widehat{L^{p}}(\mathbb{Z}^{d}):=\left\{u:\Z^{d}\to \C\Big| \|\mathcal{F}_{disc}(u)\|_{L^{p'}(\T^{d})}<\infty\right\},
         	\end{equation*}
         	with the corresponding norm 
         	\begin{equation*}
         	\|u\|_{\widehat{L^{p}}(\mathbb{Z}^{d})}:=\|\mathcal{F}_{disc}(u)\|_{L^{p'}(\T^{d})}.
         	\end{equation*}
         \end{defi}
         \begin{rem}
         	From the Hausdorff--Young inequality, we see $\ell^{p}(\Z^{d})\subseteq \widehat{L^{p}}(\Z^{d})$, if $1\le p\le2$. On the other hand, we also have $\widehat{L^{p}}(\Z^{d})\subseteq\widehat{L^{q}}(\Z^{d})\subseteq\widehat{L^{2}}(\Z^{d})=\ell^{2}(\Z^{d})$, if $1\le p\le q\le 2$, from H\"{o}lder's inequality.
         \end{rem}
         \begin{rem}
            In fact, the Fourier--Lebesgue space has a deep connection with the well--known $X_{s,b}^{r}$ space, which is firstly introduced by Bourgain (see \cite{27}) and extensively studied and used in the well--posedness theory of nonlinear dispersive equations (see \cite{28,29}). A simple example of such a connection is the following inequality 
            \begin{equation*}
            	\|e^{it\phi(D)}u_{0}\|_{Y}\lesssim \|u_{0}\|_{\widehat{L^{r}}}, \quad \forall u_{0}\in \widehat{L^{r}},
            \end{equation*}
            can imply the following inequality, for all $b>\frac{1}{r}$,
            \begin{equation*}
            	\|u\|_{Y}\lesssim \|u\|_{X_{0,b}^{r}}, \quad \forall u\in X_{0,b}^{r},
            \end{equation*}
            where $Y$ is some relatively arbitrary Banach space (see \cite{30}).
         \end{rem}
         To connect our Fourier restriction estimate for surfaces/curves of finite type $3$ with the Strichartz estimates for DNLS, we make the following important observations.
         
          The propagator $e^{it\Delta_{disc}}$ can be more explicitly written as 
         \begin{equation*}
         	\lbrace e^{it\Delta_{disc}}f(k)\rbrace_{k\in \Z^{d}}=\left\{\frac{1}{(2\pi)^{d}}\int_{\T^{d}}e^{ik\xi}e^{-it\omega(\xi)}\mathcal{F}_{disc}f(\xi)d\xi\right\}_{k\in\Z^{d}}\triangleq E^{disc}(\mathcal{F}_{disc}f),
         \end{equation*}
         where $E^{disc}$ maps functions on $\T^{d}$ into functions on $\Z^{d}\times \R$. We may call this operator $E^{disc}$ the ``discrete extension operator" due to its interesting connections with the extension operator $E_{S}$. 
         
        \begin{lemma}\label{equi}
             The following four statements are equivalent, where $\mathcal{F}_{\R^{d}}$ denotes the Fourier transform on $\R^{d}$ and the hypersurface $ S \subseteq \T^{d}\times \R$ is defined as $S:=\lbrace (\xi,t)\in \T^{d}\times \R |\; t=-\omega(\xi)\rbrace$, equipped with surface measure $d\sigma$.
             \begin{equation*}
             	(I):  \quad \left\|E^{disc} f\right\|_{L_{x,t}^{p}(\Z^{d}\times\R)}\lesssim \|f\|_{L^{q}(\T^{d})}, \quad \forall f\in L^{q}(\T^{d});
             \end{equation*}
             \begin{equation*}
             	(II): \quad \|\mathcal{F}_{\R}\mathcal{F}_{disc} g\|_{L^{q'}(S;d\sigma)}\lesssim \|g\|_{L_{x,t}^{p'}(\Z^{d}\times \R)}, \quad \forall g\in L^{p'}_{x,t}(\Z^{d}\times\R);
             \end{equation*}
             \begin{equation*}
             	(III): \quad \|E_{S} F\|_{L_{x,t}^{p}(\R^{d}\times\R)}\lesssim \|F\|_{L^{q}(\T^{d})}, \quad \forall F\in L^{q}(\T^{d});
             \end{equation*}
             \begin{equation*}
                (IV): \quad \|\mathcal{F}_{\R}\mathcal{F}_{\R^{d}} G\|_{L^{q'}(S;d\sigma)}\lesssim \|G\|_{L_{x,t}^{p'}(\R^{d}\times \R)}, \quad \forall G\in L^{p'}_{x,t}(\R^{d}\times\R).
             \end{equation*}
        \end{lemma}
        \begin{proof}
        	From duality and Plancherel's theorem, we can easily verify that (I) $\Leftrightarrow$ (II) and (III) $\Leftrightarrow$ (IV). It remains to verify that (I) $\Rightarrow$ (III) and (IV) $\Rightarrow$ (II).
        	
        	For (I) $\Rightarrow$ (III), we can express $\|E_{S}F\|_{L_{x,t}^{p}(\R^{d}\times\R)}$ as the average of translations.
        	\begin{equation*}
        		\|E_{S}F\|_{L_{x,t}^{p}(\R^{d}\times\R)}^{p}=\int_{\R}\int_{\R^{d}}\bigg|\int_{\T^{d}}e^{ix\xi}e^{-it\omega(\xi)}F(\xi)d\xi\bigg|^{p}dxdt
        	\end{equation*}
        	\begin{equation*}
        		=\int_{\R}dt\sum_{k\in \Z^{d}}\int_{\T^{d}+\lbrace k\rbrace}\bigg|\int_{\T^{d}}e^{ix\xi}e^{-it\omega(\xi)}F(\xi)d\xi\bigg|^{p}dx
        	\end{equation*}
        	\begin{equation*}
        		=\int_{\R}dt\sum_{k\in \Z^{d}}\int_{\T^{d}}\bigg|\int_{\T^{d}}e^{ix\xi}e^{-it\omega(\xi)}e^{ik\xi}F(\xi)d\xi\bigg|^{p}dx
        	\end{equation*}
        	\begin{equation*}
        		=\int_{\T^{d}}dx\int_{\R}\sum_{k\in \Z^{d}}\bigg|\int_{\T^{d}}e^{ik\xi}e^{-it\omega(\xi)}e^{ix\xi}F(\xi)d\xi\bigg|^{p}dt
        	\end{equation*}
        	\begin{equation*}
        		=(2\pi)^{dp}\int_{\T^{d}}\|E^{disc}(e^{ix\xi}F)\|_{L_{x,t}^{p}(\Z^{d}\times\R)}^{p}dx\lesssim \|F\|_{L^{q}(\T^{d})}^{p}.
        	\end{equation*}
        	For (IV) $\Rightarrow$ (II), we can take $G:=\sum_{k\in \Z^{d}}\delta_{k}(x)g(k,t), k\in \Z^{d}$, where $\delta_{k}$ is the Dirac measure at $k$. Strictly, we can replace it with $\sum_{k\in \Z^{d}}\varphi_{\varepsilon}(x-k)g(k,t)$ and let $\varepsilon\to 0$, where $\lbrace\varphi_{\varepsilon}\rbrace_{\varepsilon>0}$ is an approximation to the identity. Then, we directly see that 
        	\begin{equation*}
        		\|\mathcal{F}_{\R}\mathcal{F}_{\R^{d}} G\|_{L^{q'}(S;d\sigma)}=\|\mathcal{F}_{\R}\mathcal{F}_{disc} g\|_{L^{q'}(S;d\sigma)},
        	\end{equation*}
        	\begin{equation*}
        		\|G\|_{L^{p'}_{x,t}(\R^{d}\times \R)}=\|g\|_{L^{p'}_{x,t}(\Z^{d}\times \R)}.
        	\end{equation*}
        \end{proof}
        \begin{rem}
        	From the proof, the equivalence among these four statements actually holds for any compact, smooth hypersurface in $\T^{d}\times \R$. Therefore, we can freely reduce our discussions to the classical cases. 
        \end{rem}
        Using Lemma \ref{equi}, we can apply Theorem \ref{main}, Theorem \ref{main2} and Theorem \ref{2D} to improve the Strichartz estimate (\ref{app}).
        \begin{coro}\label{res+stri}
        	For $d=1$, $(\frac{1}{p},\frac{1}{q})$ belongs to the region in Figure 3, we have the following improved Strichartz estimate
        	\begin{equation*}
        		\|e^{it\Delta_{disc}}f\|_{L_{x,t}^{p}(\Z\times\R)}\lesssim \|f\|_{\widehat{L^{q'}}(\Z)}.
        	\end{equation*}
        	Similarly, for $d=2$, $(\frac{1}{p},\frac{1}{q})$ belongs to the region in Figure 2, we have the following improved Strichartz estimate
        	\begin{equation*}
        		\|e^{it\Delta_{disc}}f\|_{L_{x,t}^{p}(\Z^{2}\times\R)}\lesssim \|f\|_{\widehat{L^{q'}}(\Z^{2})}.
        	\end{equation*}
        \end{coro} 
        \begin{proof}
        	By Lemma \ref{equi}, it suffices to reduce $S:=\lbrace (\xi,t)\in \T^{d}\times \R |\; t=-\omega(\xi)\rbrace$ to some finite-type curves/surfaces.
        	
        	For $d=1$, we can reduce it to 
        	\begin{equation*}
        		S_{1}:=\left\lbrace (\xi,t)\in [-\frac{\pi}{4},\frac{\pi}{4}]\times \R \bigg|\; t=\sin(\xi)\right\rbrace,
        	\end{equation*}
        	\begin{equation*}
        		S_{2}:=\left\lbrace (\xi,t)\in [-\frac{\pi}{4},\frac{\pi}{4}]\times \R \bigg|\; t=\cos(\xi)\right\rbrace,
        	\end{equation*}
        	since the restriction estimate is invariant under translation and the surface is periodic.
        	
        	Note that adding a linear term to the phase is harmless for the restriction estimate, since one can remove it by a suitable linear change of spatial variables. Thus, we can reduce the problem to the restriction estimate for 
        	\begin{equation*}
        		S_{1}':=\left\lbrace (\xi,t)\in [-\frac{\pi}{4},\frac{\pi}{4}]\times \R \bigg|\; t=\sin(\xi)-\xi\right\rbrace,
        	\end{equation*}
        	\begin{equation*}
        		S_{2}':=\left\lbrace (\xi,t)\in [-\frac{\pi}{4},\frac{\pi}{4}]\times \R \bigg|\; t=\cos(\xi)-1\right\rbrace.
        	\end{equation*}
        	Now, direct calculations show that $S_{1}'$ is of finite type $3$, $S_{2}'$ is of finite type $2$, and therefore we can apply Theorem \ref{2D}.
        	
        	For $d=2$, we can also reduce $S$ into the following three cases.
        	 \begin{equation*}
        	 	S_{1}':=\left\lbrace (\xi,t)\in [-\frac{\pi}{4},\frac{\pi}{4}]^{2}\times \R \bigg|\; t=\left[\sin(\xi_{1})-\xi_{1}\right]+\left[\sin(\xi_{2})-\xi_{2}\right]\right\rbrace,
        	 \end{equation*}
        	 \begin{equation*}
        	 	S_{2}':=\left\lbrace (\xi,t)\in [-\frac{\pi}{4},\frac{\pi}{4}]^{2}\times \R \bigg|\; t=\left[\cos(\xi_{1})-1\right]+\left[\sin(\xi_{2})-\xi_{2}\right]\right\rbrace,
        	 \end{equation*}
        	 \begin{equation*}
        	 	S_{3}':=\left\lbrace (\xi,t)\in [-\frac{\pi}{4},\frac{\pi}{4}]^{2}\times \R \bigg|\; t=\left[\cos(\xi_{1})-1\right]\pm\left[\cos(\xi_{2})-1\right]\right\rbrace.
        	 \end{equation*}
        	 Then, $S_{1}'$ is of finite type $3$, and we can apply Theorem \ref{main} and Theorem \ref{main2}. For $S_{2}'$, it is mixed with finite types $2$ and $3$, but we can also have the same restriction estimate; see Remark \ref{32} for more details. Besides, $S_{3}'$ is of finite type $2$, for which the restriction estimate has already been established in Theorem \ref{WW} and Theorem \ref{Hyper}.
        \end{proof}
		 As a direct application of Corollary \ref{res+stri}, we can establish the local well--posedness (LWP) of DNLS in the Fourier--Lebesgue regime. In the following, we will assume that $p,q$ are in the range stated in Corollary \ref{res+stri}, if there is no specification.
		 
		 \begin{thm}\label{LWP}
		 	If $\alpha<p\le \alpha q', \;q\ge 2$, then there exists $T=T(\alpha,p,q,\|f\|_{\widehat{L^{q'}}})>0$, such that DNLS (\ref{nonlinear}) has a unique solution $u\in C\left([0,T];\widehat{L^{q'}}\right)\cap L^{p}\left([0,T];\ell^{p}\right)$. 
		 	
		 	If $\alpha=p, \; q\ge 2$, then there exists $\eta_{0}=\eta_{0}(p,q)>0$, such that if $0<\eta \le \eta_{0}$, and $I$ is a compact interval containing $0$ satisfying
		 	\begin{equation}\label{con}
		 		\|e^{it\Delta_{disc}}f\|_{L^{p}\left(I;\ell^{p}\right)}\le\eta,
		 	\end{equation} 
		 	then DNLS (\ref{nonlinear}) has a unique solution $u\in C\left(I;\widehat{L^{q'}}\right)\cap L^{p}\left(I;\ell^{p}\right)$.
		 \end{thm}
		 \begin{proof}
		 	We first deal with the latter case. Consider the space $X=X_{1}\cap X_{2}$ as follows,
		     \begin{equation*}
		     	X_{1}:=\left\{u\in C\left(I;\widehat{L^{q'}}\right): \|u\|_{C\left(I;\widehat{L^{q'}}\right)}\le 2\|f\|_{\widehat{L^{q'}}}+C(2\eta)^{p}\right\},
		     \end{equation*} 
		     \begin{equation*}
		     	X_{2}:=\left\{u\in L^{p}\left(I;\ell^{p}\right): \|u\|_{L^{p}\left(I;\ell^{p}\right)}\le \min\lbrace2\eta, 2C\|f\|_{\widehat{L^{q'}}}\rbrace\right\},
		     \end{equation*}
		     where $C$ is the implicit constant in the improved Strichartz estimate in Corollary \ref{res+stri} and $X$ is equipped with the norm
		     \begin{equation*}
		     	\|u\|_{X}:=\|u\|_{L^{p}\left(I;\ell^{p}\right)}+\|u\|_{C\left(I;\widehat{L^{q'}}\right)}.
		     \end{equation*}
		     We now show that the solution map $u\mapsto \Phi(u)$ defined by
		     \begin{equation*}
		        \Phi(u)(t):=e^{it\Delta_{disc}}f-i\mu\int_{0}^{t}e^{i(t-s)\Delta_{disc}}|u(s)|^{\alpha-1}u(s)ds,
		     \end{equation*}
		     is a contraction in $X$.
		     
		     From the improved Strichartz estimate, H\"{o}lder's inequality, and Minkowski's inequality, we can obtain 
		     \begin{itemize}
		     	\item 
		     	\begin{equation*}
		     		\|\Phi(u)\|_{C\left(I;\widehat{L^{q'}}\right)}\le \|f\|_{\widehat{L^{q'}}}+\int_{I}\||u(s)|^{p-1}u(s)\|_{\widehat{L^{q'}}}ds
		     	\end{equation*}
		     	\begin{equation*}
		     		\le \|f\|_{\widehat{L^{q'}}}+\int_{I}\|u(s)\|_{\ell^{p q'}}^{p}ds\le \|f\|_{\widehat{L^{q'}}}+\|u\|_{L^{p}\left(I;\ell^{p}\right)}^{p}\le \|f\|_{\widehat{L^{q'}}}+(2\eta)^{p}
		     	\end{equation*}
		     	\item 
		     	\begin{equation*}
		     		\|\Phi(u)\|_{L^{p}\left(I;\ell^{p}\right)}\le 	\|e^{it\Delta_{disc}}f\|_{L^{p}\left(I;\ell^{p}\right)}+C\int_{I}\||u(s)|^{p-1}u(s)\|_{\widehat{L^{q'}}}ds
		     	\end{equation*}
		     	\begin{equation*}
		     		\le \eta+C\|u\|_{L^{p}\left(I;\ell^{p}\right)}^{p}\le \eta+C(2\eta)^{p}\le 2\eta
		     	\end{equation*}
		     	\begin{equation*}
		     		\|\Phi(u)\|_{L^{p}\left(I;\ell^{p}\right)}\le C\|f\|_{\widehat{L^{q'}}}+C\int_{I}\||u(s)|^{p-1}u(s)\|_{\widehat{L^{q'}}}ds
		     	\end{equation*}
		     	\begin{equation*}
		     		\le C\|f\|_{\widehat{L^{q'}}}+C\|u\|_{L^{p}\left(I;\ell^{p}\right)}^{p}\le C\|f\|_{\widehat{L^{q'}}}+C(2\eta)^{p-1}(2C\|f\|_{\widehat{L^{q'}}})\le2C\|f\|_{\widehat{L^{q'}}} 
		     	\end{equation*}
		     \end{itemize}
		     The above estimates hold if we take $\eta_{0}$ sufficiently small. Therefore, we have proved that $\Phi$ maps $X$ into $X$. Next, we check $\Phi$ is a contraction.
		     \begin{itemize}
		     	\item 
		     	\begin{equation*}
		     		\|\Phi(u)-\Phi(v)\|_{C\left(I;\widehat{L^{q'}}\right)}\le \int_{I}\||u(s)|^{p-1}u(s)-|v(s)|^{p-1}v(s)\|_{\widehat{L^{q'}}}ds
		     	\end{equation*}
		     	\begin{equation*}
		     		\lesssim \int_{I}\left\{\|u(s)\|_{\ell^{p}}^{p-1}+\|v(s)\|_{\ell^{p}}^{p-1}\right\}\|u(s)-v(s)\|_{\ell^{p}}ds
		     	\end{equation*}
		     	\begin{equation*}
		     		\lesssim \left(\|u\|_{L^{p}\left(I;\ell^{p}\right)}^{p-1}+\|v\|_{L^{p}\left(I;\ell^{p}\right)}^{p-1}\right)\|u-v\|_{L^{p}\left(I;\ell^{p}\right)}\le 2(2\eta)^{p-1}\|u-v\|_{X}
		     	\end{equation*}
		     	\item 
		     	\begin{equation*}
		     		\|\Phi(u)-\Phi(v)\|_{L^{p}\left(I;\ell^{p}\right)}\lesssim \int_{I}\||u(s)|^{p-1}u(s)-|v(s)|^{p-1}v(s)\|_{\widehat{L^{q'}}}ds
		     	\end{equation*}
		     	\begin{equation*}
		     		\lesssim 2(2\eta)^{p-1}\|u-v\|_{X}
		     	\end{equation*}
		     \end{itemize}
		     Thus, if $\eta_{0}$ is sufficiently small, we have $\|\Phi(u)-\Phi(v)\|_{X}\le \frac{1}{2}\|u-v\|_{X}$, and the fixed point yields a unique solution to DNLS (\ref{nonlinear}).
		    
		     \vspace{5pt}
		     Next, we complete the proof by establishing the first statement, which is easier than the latter case.
		     
		     We denote $R:=\|f\|_{\widehat{L^{q'}}}$, and then we consider the closed ball $B$ with radius $2CR$ in the space $C\left([0,T];\widehat{L^{q'}}\right)\cap L^{p}\left([0,T];\ell^{p}\right)$, which is equipped with the norm
		     \begin{equation*}
		     	\|u\|_{B}:=\|u\|_{L^{p}\left([0,T];\ell^{p}\right)}+\|u\|_{C\left([0,T];\widehat{L^{q'}}\right)}.
		     \end{equation*}
		     With the same solution map $\Phi$, we quickly check that $\Phi$ is a contraction on $B$. For convenience, we assume that $C>10$.
		     \begin{itemize}
		     	\item 
		     	\begin{equation*}
		     		\|\Phi(u)\|_{C\left([0,T];\widehat{L^{q'}}\right)}\le \|f\|_{\widehat{L^{q'}}}+\int_{0}^{T}\|u(s)\|_{\ell^{p}}^{\alpha}ds
		     	\end{equation*}
		     	\begin{equation*}
		     	   \le \|f\|_{\widehat{L^{q'}}}+T^{\frac{p-\alpha}{p}}\|u\|_{L^{p}\left([0,T];\ell^{p}\right)}^{\alpha}\le R+T^{\frac{p-\alpha}{p}}(2CR)^{\alpha}\le \frac{1}{2}CR
		     	\end{equation*}
		     	\item 
		     	\begin{equation*}
		     		\|\Phi(u)\|_{L^{p}\left([0,T];\ell^{p}\right)}\le C\|f\|_{\widehat{L^{q'}}}+C\int_{0}^{T}\|u(s)\|_{\ell^{p}}^{\alpha}ds
		     	\end{equation*}
		     	\begin{equation*}
		     		\le CR+CT^{\frac{p-\alpha}{p}}R^{\alpha}\le \frac{3}{2}CR
		     	\end{equation*}
		     	\item 
		     	\begin{equation*}
		     		\|\Phi(u)-\Phi(v)\|_{C\left([0,T];\widehat{L^{q'}}\right)}\le \int_{0}^{T}\||u(s)|^{\alpha-1}u(s)-|v(s)|^{\alpha-1}v(s)\|_{\widehat{L^{q'}}}ds
		     	\end{equation*}
		     	\begin{equation*}
		     		\lesssim \int_{0}^{T}\left\{\|u(s)\|_{\ell^{p}}^{\alpha-1}+\|v(s)\|_{\ell^{p}}^{\alpha-1}\right\}\|u(s)-v(s)\|_{\ell^{p}}ds
		     	\end{equation*}
		     	\begin{equation*}
		     		\lesssim T^{\frac{p-\alpha}{p}}\left(\|u\|_{L^{p}\left([0,T];\ell^{p}\right)}^{\alpha-1}+\|v\|_{L^{p}\left([0,T];\ell^{p}\right)}^{\alpha-1}\right)\|u-v\|_{L^{p}\left([0,T];\ell^{p}\right)}\le 2T^{\frac{p-\alpha}{p}}R^{\alpha-1}\|u-v\|_{B}
		     	\end{equation*}
		     	\item 
		     	\begin{equation*}
		     		\|\Phi(u)-\Phi(v)\|_{L^{p}\left([0,T];\ell^{p}\right)}\lesssim \int_{0}^{T}\||u(s)|^{\alpha-1}u(s)-|v(s)|^{\alpha-1}v(s)\|_{\widehat{L^{q'}}}ds
		     	\end{equation*}
		     	\begin{equation*}
		     		\lesssim 2T^{\frac{p-\alpha}{p}}R^{\alpha-1}\|u-v\|_{B}
		     	\end{equation*}
		     \end{itemize}
		     Then, we can take $T$ sufficiently small, and the contraction is ensured.
		 \end{proof}
		\begin{rem}
			From the proof, we can derive the following continuation criterion.
			If the maximal existence time $T^{\ast}<\infty$, then 
			\begin{equation*}
				\left\{
				\begin{aligned}
					& \sup_{t\in[0,T^{\ast})}\|u(t)\|_{\widehat{L^{q'}}}=\infty,  \quad \text{if} \; \alpha<p\le \alpha q',\\
					& \|u\|_{L^{p}\left([0,T^{\ast});\ell^{p}\right)}=\infty, \quad \text{if} \; \alpha=p.
				\end{aligned}
				\right.
			\end{equation*}
			Besides, when $q\ge 2$, we have $\widehat{L^{q'}}\subseteq\ell^{2}$, so we actually have established the well--posedness theory with higher integrability.
		\end{rem}
			In the latter case, i.e., $\alpha=p$, we see that if the $\widehat{L^{q'}}$--norm of initial data $f$ is sufficiently small, then we actually can establish the GWP of DNLS in the $\widehat{L^{q'}}$--regime. One step further, we can immediately obtain the asymptotic completeness of DNLS in the $\widehat{L^{q'}}$--regime.
		\begin{coro}
			For $d=1$, $q\ge2$ and $\alpha$ satisfying 
			\begin{equation*}
				\left\{
				\begin{aligned}
					& 0<\frac{1}{\alpha}\le \frac{1}{4q'},  \quad \text{if} \;\; 0<\frac{1}{q}\le\frac{1}{2},\\
					& 0<\frac{1}{\alpha}< \frac{1}{4}, \quad \text{if} \;\; q=\infty,
				\end{aligned}
				\right.
			\end{equation*}
			 there exists $\varepsilon_{0}>0$, such that if $\|f\|_{\widehat{L^{q'}}}\le \varepsilon_{0}$, then DNLS (\ref{nonlinear}) has a unique global solution $u\in C\left(\R;\widehat{L^{q'}}\right)\cap L^{\alpha}\left(\R;\ell^{\alpha}\right)$.
			
			For $d=2$, $q\ge 2$ and $\alpha$ satisfying 
			\begin{equation*}
				\left\{
				\begin{aligned}
					& 0<\frac{1}{\alpha}<\frac{4}{13},  \quad \text{if} \;\; 0<\frac{1}{q}\le\frac{3}{130},\\
					& 0<\frac{1}{\alpha}< \frac{4}{13}-\frac{7}{31}\left(\frac{1}{q}-\frac{3}{130}\right), \quad \text{if} \;\; \frac{3}{130}<\frac{1}{q}\le\frac{1}{2},\\
					& 0<\frac{1}{\alpha}\le \frac{1}{5}, \quad \text{if} \;\; q=2,
				\end{aligned}
				\right.
			\end{equation*}
			there also exists $\varepsilon_{0}>0$, such that if $\|f\|_{\widehat{L^{q'}}}\le \varepsilon_{0}$, then DNLS (\ref{nonlinear}) has a unique global solution $u\in C\left(\R;\widehat{L^{q'}}\right)\cap L^{\alpha}\left(\R;\ell^{\alpha}\right)$.
			
			Besides, we have the asymptotic completeness in the $\widehat{L^{q'}}$--regime, i.e., for the solution $u$ above, we can find $u_{\pm}\in \widehat{L^{q'}}$ such that
			\begin{equation*}
				\|u(t)-e^{it\Delta_{disc}}u_{\pm}\|_{\widehat{L^{q'}}}\to 0, \quad t\to\pm \infty,
			\end{equation*}
			or equivalently,
			\begin{equation*}
				\|e^{-it\Delta_{disc}}u(t)-u_{\pm}\|_{\widehat{L^{q'}}}\to 0, \quad t\to\pm \infty.
			\end{equation*}
		\end{coro}
	\begin{proof}
		The GWP follows from the improved Strichartz estimate in Corollary \ref{res+stri} and Theorem \ref{LWP} by choosing $I=\R$. As for the asymptotic completeness, we have the identity
		\begin{equation*}
			e^{-it\Delta_{disc}}u=f-i\mu \int_{0}^{t}e^{-is\Delta_{disc}}|u(s)|^{\alpha-1}u(s)ds.
		\end{equation*}
		Therefore, it suffices to show that
		\begin{equation*}
			\left\|\int_{t}^{\pm\infty}e^{-is\Delta_{disc}}|u(s)|^{\alpha-1}u(s)ds\right\|_{\widehat{L^{q'}}}\to 0, \quad t\to \pm\infty,
		\end{equation*}
		which follows from the monotone convergence and the estimate
		\begin{equation*}
			\left\|\int_{\R}e^{-is\Delta_{disc}}|u(s)|^{\alpha-1}u(s)ds\right\|_{\widehat{L^{q'}}}\le \int_{\R}\left\||u(s)|^{\alpha-1}u(s)\right\|_{\widehat{L^{q'}}}ds\le \|u\|_{L^{\alpha}(\R;\ell^{\alpha})}^{\alpha}<\infty.
		\end{equation*}
	\end{proof}

		\newpage
		\section*{Acknowledgement}
		The author is grateful to Prof. Alex Cohen for helpful introductions and discussions, and to Zhuoran Li for pointing out that the original result can be improved and the proof can be simplified.
		
		\section*{Conflict of interest statement}
		The author does not have any possible conflict of interest.
		
		\section*{Data availability statement}
		The manuscript has no associated data.
		\bigskip
		\bigskip

		\bibliographystyle{alpha}
		\bibliography{restriction}

	\end{document}